\documentclass{amsart}
\usepackage{amssymb}
\usepackage{amsmath}
\usepackage{amsthm}
\usepackage{mathrsfs}
\newtheorem{theorem}{Theorem}
\newtheorem{lemma}[theorem]{Lemma}
\newtheorem{assumption}[theorem]{Assumption}

\theoremstyle{definition}
\newtheorem{definition}[theorem]{Definition}

\newtheorem{proposition}[theorem]{Proposition}

\makeatletter

\@addtoreset{theorem}{section}
\makeatother

\makeatletter

\@addtoreset{equation}{section}
\makeatother

\begin{document}                                                 
\title[Advection-Diffusion Equation on Evolving Surface with Boundary]{Local and Global Solvability for Advection-Diffusion Equation on an Evolving Surface with a Boundary}                                 
\author[Hajime Koba]{Hajime Koba}                                
\address{Graduate School of Engineering Science, Osaka University,\\
1-3 Machikaneyamacho, Toyonaka, Osaka, 560-8531, Japan}                                  
%\curraddr{...}                                   
\email{iti@sigmath.es.osaka-u.ac.jp}                                      
%\urladdr{...}                                    
%\dedicatory{...}                                 
\date{}                                       
%\thanks{...}                                     
%\translator{...}                                 
\begin{abstract}
This paper considers the existence of local and global-in-time strong solutions to the advection-diffusion equation with variable coefficients on an evolving surface with a boundary. We apply both the maximal $L^p$-in-time regularity for Hilbert space-valued functions and the semigroup theory to construct local and global-in-time strong solutions to an evolution equation. Using the approach and our function spaces on the evolving surface, we show the existence of local and global-in-time strong solutions to the advection-diffusion equation. Moreover, we derive the asymptotic stability of the global-in-time strong solution. 
\end{abstract}       
\maketitle

\section{Introduction}\label{sect1}
We are interested in the existence of local and global-in-time solutions to the advection-diffusion equation on an evolving surface with a boundary. An evolving surface means that the surface is moving or the shape of the surface is changing along with the time. This paper has three purposes. The first one is to construct a strong solution of an evolution equation by applying both the maximal $L^p$-in-time regularity for Hilbert space-valued functions and the semigroup theory. The second one is to introduce and study function spaces on evolving surfaces. The third one is to apply our approach and function spaces to show the existence of local and global-in-time strong solutions to the advection-diffusion equation with variable coefficients on an evolving surface with a boundary.

Let us first introduce basic notations. Let $x = { }^t ( x_1 , x_2, x_3) \in \mathbb{R}^3$, $X = { }^t (X_1 , X_2) \in \mathbb{R}^2$ be the spatial variables, and $t , \tau \geq 0$ be the time variables. The symbols $\nabla$ and $\nabla_X$ are two gradient operators defined by $\nabla = { }^t (\partial_1 , \partial_2 , \partial_3)$ and $\nabla_X = { }^t (\partial_{X_1} , \partial_{X_2})$, where $\partial_j = \partial/{\partial x_j}$ and $\partial_{X_\alpha} = \partial /{ \partial X_\alpha}$. Let $T \in (0, \infty]$, and let $\Gamma (t) (= \{ \Gamma (t) \}_{0 \leq t < T}) $ be an evolving surface with a boundary $\partial \Gamma (t)$ such that for $0 \leq t < T$, $\Gamma (t) = \{ x \in \mathbb{R}^3; { \ }x = \widehat{x}(X,t), X \in U \}$, where $U \subset \mathbb{R}^2$ is a bounded domain with a $C^3$-boundary $\partial U$ and $\widehat{x} = { }^t ( \widehat{x}_1, \widehat{x}_2 , \widehat{x}_3) \in [C^3 ( \overline{U} \times [0, T ))]^3$.

This paper considers the existence of local and global-in-time strong solutions to the \emph{advection-diffusion equation on the evolving surface} $\Gamma (t)$:
\begin{equation}\label{eq11}
\begin{cases}
D_t^w u + ({\rm{div}}_\Gamma w ) u - {\rm{div}}_\Gamma ( \kappa \nabla_\Gamma u ) = 0& \text{ on } \Gamma_T,\\
u|_{\partial \Gamma_T} = 0,\\
u|_{t = 0} = u_0 & \text{ on } \Gamma_0 ,
\end{cases}
\end{equation}
where the unknown function $u = u (x,t)$ is the \emph{concentration of amount of a substance} on $\Gamma (t)$, while the given function $w = w (x,t) = { }^t (w_1,w_2,w_3)$ is the \emph{motion velocity} of $\Gamma (t)$, the given function $\kappa = \kappa (x,t)$ is the \emph{diffusion coefficient}, and $u_0 = u_0 (x)$ is the given \emph{initial datum}. Here $D_t^w u := \partial_t u + ( w \cdot \nabla ) u$, $\partial_t := \partial/{\partial t}$, $\Gamma_0 := \Gamma (0)$,
\begin{equation*}
\Gamma_T := \bigcup_{0 < t < T}\{ \Gamma (t) \times \{ t \} \}, \text{ and }\partial \Gamma_T := \bigcup_{0 < t < T}\{ \partial \Gamma (t) \times \{ t \} \}.
\end{equation*}
See Section \ref{sect2} for the notations $w$, ${\rm{div}}_\Gamma$, and $\nabla_\Gamma$.

Let us state mathematical analysis of system \eqref{eq11}. Dziuk-Elliott \cite{DE07} applied a Galerkin-type method to show the existence of a unique weak solution to system \eqref{eq11} on an evolving closed surface when $\kappa = 1$ and $u_0 \in W^{1,2}( \Gamma_0 )$. Alphonse-Elliott-Stinner \cite{AES15a}, \cite{AES15b} introduced their function spaces on evolving surfaces to study several parabolic PDEs on evolving surfaces. They applied their evolving Hilbert space to construct a weak solution of several PDEs on evolving surfaces. Djurdjevac \cite{Dju17} and Djurdjevac-Elliott-Kornhuber-Ranner \cite{DEKR18} studied system \eqref{eq11} with random coefficients on an evolving closed surface. They improved the approaches in \cite{AES15a}, \cite{AES15b} to show the existence of a unique mean-weak solution to system \eqref{eq11} when $u_0 \in W^{1,2} ( \Gamma_0)$. This paper attacks the existence of strong solutions to \eqref{eq11} with variable coefficients on an evolving surface with a boundary when $u_0 \in W^{1,2}_0 ( \Gamma_0 )$.

To study system \eqref{eq11}, we set $v = v (X,t) = u (\widehat{x} (X,t) ,t)$, $v_0 = v_0 (X) = u_0 ( \widehat{x} (X,0))$, and $\widehat{\kappa} = \widehat{\kappa} (X, t) = \kappa ( \widehat{x} (X,t) ,t)$. Then $v$ satisfies
\begin{equation}\label{eq12}
\begin{cases}
\partial_t v + L v = 0 & \text{ in } U \times (0,T),\\
v|_{\partial U} = 0 & \text{ on } (0,T),\\
v|_{t = 0} = v_0 & \text{ in } U.
\end{cases}
\end{equation}
Here
\begin{equation}\label{eq13}
L f = L (t) f := - \sum_{\alpha , \beta =1}^2 \frac{1}{\sqrt{ \mathcal{G}}} \frac{\partial}{\partial X_\alpha} \left\{ \widehat{\kappa} \sqrt{ \mathcal{G}} \mathfrak{g}^{\alpha \beta } \frac{\partial f }{\partial X_\beta} \right\} + \frac{1}{2 \mathcal{G}} \left( \frac{d \mathcal{G}}{d t} \right) f.
\end{equation}
See Section \ref{sect2} for the notations $\mathcal{G}$ and $\mathfrak{g}^{\alpha \beta}$. We call $L$ the \emph{Laplace-Beltrami operator} when $\widehat{\kappa} = 1$ and ${d \mathcal{G}}/{d t} = 0$. This paper constructs a strong solution to system \eqref{eq12} to show the existence of a strong solution to system \eqref{eq11}. Notice that system \eqref{eq11} is equivalent to system \eqref{eq12} by Definition \ref{def21} and an inverse function theorem when solutions to systems \eqref{eq11} and \eqref{eq12} are smooth functions. See Definition \ref{def22} and Section \ref{sect5} for the strong solutions to system \eqref{eq11}.

To construct solutions of system \eqref{eq12}, we have to deal with the term $ \widehat{\kappa} \mathfrak{g}^{\alpha \beta} \frac{\partial^2 f }{\partial X_\alpha \partial X_\beta}$ of \eqref{eq13}. However, it is not easy to derive some properties of the term since $\widehat{\kappa}$ and $\mathfrak{g}^{\alpha \beta}$ depend on both $t$ and $X$. This is the main difficulty in constructing strong solutions to \eqref{eq12}.

Let us explain the two key ideas for constructing strong solutions to system \eqref{eq12}. The first one is to apply some properties of the following linear operator:
\begin{equation}\label{eq14}
\begin{cases}
A f := - \left( \lambda_1 \frac{\partial^2}{\partial X_1^2} + \lambda_2 \frac{\partial^2}{\partial X_2^2} \right) f ,\\
D (A ) := L^2 (U) \cap W^{1,2}_0(U) \cap W^{2,2} (U). 
\end{cases}
\end{equation}
Here $\lambda_1 , \lambda_2 >0$. More precisely, we use the elliptic regularity of $A$, i.e., there is $C_\sharp >0$ such that for all $f \in D (A)$
\begin{equation}\label{eq15}
\| f \|_{L^2 (U)} + \| \nabla_X f \|_{L^2(U)} + \| \nabla_X^2 f \|_{L^2(U)} \leq C_\sharp \| A f \|_{L^2 (U)},
\end{equation}
and the maximal $L^2$-regularity of $A$, i.e., there is $C_A >0$ such that for each $(F ,V_0 ) \in L^2 (0,T;L^2(U)) \times D (A^{\frac{1}{2}})$ there exists a unique function $V$ satisfying the system:
\begin{equation*}
\begin{cases}
\frac{d}{d t}V + A V = F & \text{ on } (0, T),\\
V|_{t =0} = V_0,
\end{cases}
\end{equation*}
and the estimate:
\begin{multline}\label{eq16}
\left( \int_0^T \left\| \frac{d V}{d t} \right\|_{L^2(U)}^2 d t \right)^{\frac{1}{2}} + \left( \int_0^T \| A V \|_{L^2(U)}^2 d t \right)^{\frac{1}{2}} \\
\leq  \sqrt{2} \| A^{\frac{1}{2}} V_0 \|_{L^2(U)} +C_A \left( \int_0^T \| F \|_{L^2(U)}^2 d t \right)^{\frac{1}{2}}. 
\end{multline}
Remark that the two positive constants $C_\sharp$, $C_A$ depend only on $(\lambda_1 , \lambda_2 , U)$. Remark also that we can write
\begin{equation*}
C_\sharp = \frac{C_\sharp'( q , U)}{\lambda_1} \text{ and } C_A= \frac{C_A'( q , U)}{\lambda_1} 
\end{equation*}
when $\lambda_2 = q \lambda_1$ for some $q >0$. See Section \ref{sect3} for details.

The second one is to consider the following approximate equations:
\begin{equation*}
\begin{cases}
\frac{d}{d t} v_1 + A v_1 = 0 \text{ on }(0,T),\\
v_1 |_{t=0 } = v_0,
\end{cases}{ \ }\begin{cases}
\frac{d}{d t} v_{m+1} + A v_{m+1} = -B v_m \text{ on }(0,T),\\
v_{m+1} |_{t=0 } = v_0.
\end{cases}
\end{equation*}
Here $B = B (t)$ is the evolution operator defined by $B = L - A$. Using the approximate equations, the maximal $L^2$-regularity of $A$, the semigroup theory, and the norm $\| \cdot \|_{Z_T}$:
\begin{equation*}
\| \varphi \|_{Z_T} := \sup_{0 < t <T}\{ \mathrm{e}^{- t} \| \varphi \|_{L^2 (U)} \} + \| d \varphi /{d t} \|_{L^2 (0,T;L^2(U))} + \| A \varphi \|_{L^2 (0,T;L^2(U))},
\end{equation*}
we construct local and global-in-time strong solutions to system \eqref{eq12}. This is the key approximation method of this paper. See Sections \ref{sect6} and \ref{sect7} for details.

The outline of this paper is as follows: In Section \ref{sect2} we state the assumptions of our evolving surface $\Gamma (t)$ and the main results of this paper. In Section \ref{sect3} we study differential operators on the evolving surface $\Gamma (t)$ and the elliptic operator $A$ defined by \eqref{eq14}. In Section \ref{sect4} we introduce and study function spaces on the evolving surface $\Gamma (t)$. In Section \ref{sect5} we study basic properties of strong solutions to the advection-diffusion equation \eqref{eq11}. In Section \ref{sect6} we construct strong solutions to an evolution equation by the maximal $L^p$-regularity and semigroup theory. In Section \ref{sect7} we show the existence of local and global-in-time strong solutions to the advection-diffusion equation \eqref{eq11} on the evolving surface $\Gamma (t)$. In the Appendix, we introduce the maximal $L^p$-regularity and the basic semigroup theory.

\section{Main Results}\label{sect2}
We first introduce the definition of our evolving surface with a boundary and notations. Then, we state the main results of this paper.
\begin{definition}[Evolving surface with boundary]\label{def21}
Let $T \in (0, \infty]$. For $0 \leq t < T$, let $\Gamma (t) \subset \mathbb{R}^3$ be a set. We call $\Gamma (t) (= \{ \Gamma (t) \}_{0 \leq t <T} )$ an \emph{evolving surface with a boundary} if the following five properties hold:\\
$(\mathrm{i})$ For every $0 \leq t < T$, $\Gamma (t)$ can be written by
\begin{equation*}
\Gamma (t) = \{ x = { }^t ( x_1 , x_2 , x_3 ) \in \mathbb{R}^3; { \ }x = \widehat{x} (X,t) , X \in U \},
\end{equation*}
where $U \subset \mathbb{R}^2$ is a bounded domain with a $C^3$-boundary $\partial U$ and $\widehat{x} = { }^t ( \widehat{x}_1 , \widehat{x}_2 , \widehat{x}_3) \in [ C^3 ( \overline{U} \times [0,T ))]^3$.\\
$(\mathrm{ii})$ For each $0 \leq t <T$,
\begin{equation*}
\widehat{x} ( \cdot , t) : \overline{U} \to \overline{\Gamma (t)} \text{ is bijective}.
\end{equation*}
$(\mathrm{iii})$ For every $0 \leq t < T$, the boundary $\partial \Gamma (t)$ of $\Gamma (t)$ can be written by
\begin{equation*}
\partial \Gamma (t) = \{ x \in \mathbb{R}^3; { \ } x = \widehat{x} (X, t ), X \in \partial U \}.
\end{equation*}
$(\mathrm{iv})$ There is $\lambda_{min} >0$ such that for every $0 \leq t < T$ and $X \in \overline{U}$,
\begin{multline*}
\left( \frac{\partial \widehat{x}_2}{\partial X_1} \frac{\partial \widehat{x}_3}{\partial X_2} - \frac{\partial \widehat{x}_2}{\partial X_2} \frac{\partial \widehat{x}_3}{\partial X_1} \right)^2 + \left( \frac{\partial \widehat{x}_1}{\partial X_1} \frac{\partial \widehat{x}_3}{\partial X_2} - \frac{\partial \widehat{x}_1}{\partial X_2} \frac{\partial \widehat{x}_3}{\partial X_1} \right)^2\\
 + \left( \frac{\partial \widehat{x}_1}{\partial X_1} \frac{\partial \widehat{x}_2}{\partial X_2} - \frac{\partial \widehat{x}_1}{\partial X_2} \frac{\partial \widehat{x}_2}{\partial X_1} \right)^2 \geq \lambda_{min} .
\end{multline*}
$(\mathrm{v})$ There is $\lambda_{max} >0$ such that for every $0 \leq t < T$, $X \in \overline{U}$, $j=1,2,3$, and $\alpha , \beta =1,2$,
\begin{equation*}
\left| \frac{\partial \widehat{x}_j}{\partial X_\alpha} \right| + \left| \frac{\partial^2 \widehat{x}_j}{\partial X_\alpha \partial X_\beta} \right| + \left| \frac{\partial^2 \widehat{x}_j}{\partial t \partial X_\alpha} \right| + \left| \frac{\partial^3 \widehat{x}_j}{\partial t \partial X_\alpha \partial X_\beta} \right| \leq \lambda_{max} < + \infty.
\end{equation*}
\end{definition}

Let us explain the conventions used in this paper. We use the Greek characters $\alpha , \beta, \zeta , \eta$ as $1, 2$. Moreover, we often use the following Einstein summation convention: $c_\alpha d_{\alpha \beta} = \sum_{\alpha = 1}^2 c_\alpha d_{\alpha \beta}$ and $c_\zeta d^{\zeta \eta} = \sum_{\zeta = 1}^2 c_\zeta d^{\zeta \eta}$. The symbol $d \mathcal{H}^k_x$ denotes the $k$-dimensional Hausdorff measure.

Next we define some notations. Let $T \in (0, \infty]$ and $\Gamma (t) (= \{ \Gamma (t) \}_{0 \leq t <T})$ be an evolving surface with a boundary. By definition, there are bounded domain $U \subset \mathbb{R}^2$ with a $C^3$-boundary and $\widehat{x} = \widehat{x} (X,t) = { }^t ( \widehat{x}_1 , \widehat{x}_2 , \widehat{x}_3 ) \in [ C^3 ( \overline{U} \times [0, T ) ) ]^3$ satisfying the properties as in Definition \ref{def21}. The symbol $n = n (x , t) = { }^t (n_1 , n_2 , n_3)$ denotes the unit outer normal vector at $x \in \overline{\Gamma (t)}$, the symbol $\nu = \nu ( x, t) = { }^t ( \nu_1 , \nu_2 , \nu_3)$ denotes the unit outer co-normal vector at $x \in \partial \Gamma (t)$, and the symbol $n^U = n^U (X) = { }^t (n_1^U , n_2^U)$ denotes the unit outer normal vector at $X \in \partial U$. Set
\begin{equation*}
\Gamma_T = \bigcup_{0 <  t  < T}\{ \Gamma (t) \times \{ t \} \} \text{ and }\overline{\Gamma_T} = \bigcup_{0 \leq  t  < T}\{ \overline{\Gamma (t)} \times \{ t \} \}.
\end{equation*}

Fix $j \in \{ 1 , 2 , 3 \}$. For $\psi \in C^1 (\mathbb{R}^3)$ and $\phi = { }^t ( \phi_1, \phi_2 , \phi_3) \in [C^1 (\mathbb{R}^3)]^3$,
\begin{align*}
\partial_j^\Gamma \psi & := \sum_{i=1}^3 (\delta_{i j} - n_i n_j ) \partial_i \psi,\\ 
\nabla_\Gamma & :=  { }^t ( \partial_1^\Gamma , \partial_2^\Gamma , \partial_3^\Gamma ),\\
{\rm{grad}}_\Gamma \psi & := \nabla_\Gamma \psi = { }^t (\partial_1^\Gamma \psi , \partial_2^\Gamma \psi , \partial_3^\Gamma \psi ) ,\\
{\rm{div}}_\Gamma \phi & := \nabla_\Gamma \cdot \phi = \partial_1^\Gamma \phi_1 + \partial_2^\Gamma \phi_2 + \partial_3^\Gamma \phi_3. 
\end{align*}
Here $\delta_{ij}$ is the Kronecker delta. We write $\widetilde{X} = \widetilde{X} (x,t)$ as the \emph{inverse mapping} of $\widehat{x} = \widehat{x}(X,t)$, i.e., for each $0 \leq t <T$,
\begin{align*}
\widetilde{X} ( \cdot , t ) &: \overline{\Gamma (t)} \to \overline{U},\\
\widehat{x}(\widetilde{X} (x, t), t) & = x \text{ for } x \in \overline{\Gamma (t)},\\
\widetilde{X}(\widehat{x}(X,t) , t) & = X \text{ for }X \in \overline{U}.
\end{align*}
Note that there exists an inverse mapping $\widetilde{X}$ of $\widehat{x}$ by Definition \ref{def21} and an inverse function theorem. Set
\begin{equation*}
w ( x,t) = \widehat{x}_t ( \widetilde{X}(x,t),t).
\end{equation*}
We call $w$ the \emph{motion velocity of the evolving surface $\Gamma (t)$}. We also call $w$ the \emph{speed of the evolving surface $\Gamma (t)$}.

We state the definition of a strong solution to system \eqref{eq11} with $u_0 \in W_0^{1,2} (\Gamma_0)$.
\begin{definition}[Strong solution to \eqref{eq11}]\label{def22}
Let $u_0 \in W^{1,2}_0 ( \Gamma_0)$ and
\begin{equation*}
u \in L^2(0,T:W^{1,2}_0 \cap W^{2,2} (\Gamma (\cdot ))) \cap W^{1,2}(0,T;L^2(\Gamma (\cdot ))).
\end{equation*}
We call $u$ a \emph{strong solution} of system \eqref{eq11} with initial datum $u_0$ if the function $u$ satisfies the following three properties:\\
$(\mathrm{i})$ (Equation)
\begin{equation*}
\left\| D_t^w u + ({\rm{div}}_\Gamma w )u - {\rm{div}}_\Gamma ( \kappa \nabla_\Gamma u ) \right\|_{L^2(0,T ; L^2( \Gamma (\cdot )))} = 0.
\end{equation*}
$(\mathrm{ii})$ (Boundary condition)
\begin{equation*}
\| \gamma u \|_{L^2 (0, T; L^2 (\partial \Gamma (\cdot )))} = 0.
\end{equation*}
Here $\gamma : W^{1,2} ( \Gamma (t)) \to L^2 ( \partial \Gamma (t) )$ is the trace operator defined by Proposition \ref{prop43}.\\
$(\mathrm{iii})$ (Initial condition) $\widehat{u} \in C ([0,T); L^2(U))$ and
\begin{equation*}
\lim_{t \to 0 + 0} \| u ( \widehat{x}(\cdot , t),t) - u_0 (\widehat{x} ( \cdot ,0)) \|_{L^2 (U)} = 0.
\end{equation*}
Here $\widehat{u} = \widehat{u} (X,t) := u( \widehat{x} (X,t) , t) $.
\end{definition}
\noindent See Section \ref{sect4} for function spaces on the evolving surface $\Gamma (t)$ and its norms. See also Section \ref{sect5} for the strong solutions of system \eqref{eq11}.

Before introducing the main results of this paper, we state the assumption of the evolving surface $\Gamma (t)$. For every $0 \leq t < T$ and $X \in \overline{U}$,
\begin{multline*}
\mathfrak{g}_1 = \mathfrak{g}_1 (X,t) := \frac{\partial \widehat{x}}{\partial X_1},{ \ }\mathfrak{g}_2 = \mathfrak{g}_2 (X,t) := \frac{\partial \widehat{x}}{\partial X_2},{ \ }\mathfrak{g}_{\alpha \beta} := \mathfrak{g}_\alpha \cdot \mathfrak{g}_\beta,\\
( \mathfrak{g}^{ \alpha \beta })_{2 \times 2} := ( \mathfrak{g}_{\alpha \beta })_{2 \times 2}^{-1} = \frac{1}{\mathfrak{g}_{11} \mathfrak{g}_{22} - \mathfrak{g}_{12} \mathfrak{g}_{21} } 
\begin{pmatrix}
\mathfrak{g}_{22} & - \mathfrak{g}_{21}\\
- \mathfrak{g}_{12} & \mathfrak{g}_{11}
\end{pmatrix},{ \ }\mathfrak{g}^\alpha := \mathfrak{g}^{\alpha \beta} \mathfrak{g}_\beta,\\
\mathcal{G} = \mathcal{G} ( X , t) := \mathfrak{g}_{11} \mathfrak{g}_{22} - \mathfrak{g}_{12} \mathfrak{g}_{21}.
\end{multline*}
Note that $\mathcal{G} = | \mathfrak{g}_1 \times \mathfrak{g}_2 |^2 =\mathfrak{g}_{11} \mathfrak{g}_{22} - \mathfrak{g}_{12} \mathfrak{g}_{21} = \lambda_{min} >0$. Write $\widehat{\kappa} = \widehat{\kappa} (X,t) = \kappa (\widehat{x} (X,t) , t)$. Throughout this paper, we assume the following restrictions.
\begin{assumption}\label{ass23}
$(\mathrm{i})$ Assume that $\kappa \in C^2 ( \overline{\Gamma_T})$. Here
\begin{equation*}
C^2 ( \overline{\Gamma_T}):= \{ \psi; { \ }\psi (x,t) = \varphi (\widetilde{X} (x,t), t ) , { \ }\varphi \in C^2 ( \overline{U} \times [0,T) ) \}.
\end{equation*}
$(\mathrm{ii})$ Assume that there are two positive constants $\kappa_{min}, \kappa_{max}$ such that for all $(X, t) \in \overline{U} \times [0, T )$ and $\alpha =1,2$,
\begin{equation*}
\kappa_{min} \leq \widehat{\kappa} \leq \kappa_{max} \text{ and }\left| \frac{\partial \widehat{\kappa}}{\partial t} \right| + \left| \frac{\partial \widehat{\kappa}}{\partial X_\alpha} \right| + \left| \frac{\partial^2 \widehat{\kappa}}{\partial t \partial X_\alpha} \right| < \kappa_{max}.
\end{equation*}
$(\mathrm{iii})$ Assume that there are two positive constants $\lambda_1, \lambda_2$ such that for all $(X, t) \in \overline{U} \times [0, T )$,
\begin{equation*}
\widehat{\kappa} \mathfrak{g}^{11} > \lambda_1 \text{ and }\widehat{\kappa} \mathfrak{g}^{12} > \lambda_2.
\end{equation*}
\end{assumption}
\noindent Remark that $\inf \mathfrak{g}^{11} \neq \inf \mathfrak{g}^{22}$ in general. For $0 \leq t < T$,
\begin{align*}
\mathcal{M}_1(t) &:= \left\| \widehat{\kappa} \frac{\mathfrak{g}_{11} }{\mathcal{G}} - \lambda_2 \right\|_{L^\infty (U)} + \left\| \widehat{\kappa} \frac{\mathfrak{g}_{22} }{\mathcal{G}} - \lambda_1 \right\|_{L^\infty (U)} + \left\| \widehat{\kappa} \frac{\mathfrak{g}_{12}}{ \mathcal{G}} \right\|_{L^\infty (U) },\\
\mathcal{M}_2 (t) & := \left\| \frac{\widehat{\kappa}}{ \mathcal{G} } \left( \frac{\partial \mathfrak{g}_{22}}{\partial X_1} - \frac{\partial \mathfrak{g}_{12}}{\partial X_2} \right) - \frac{\widehat{\kappa}}{2 \mathcal{G}^2} \left( \mathfrak{g}_{22} \frac{\partial \mathcal{G} }{\partial X_1} - \mathfrak{g}_{12} \frac{\partial \mathcal{G} }{\partial X_2} \right) \right\|_{L^\infty (U)} ,\\
\mathcal{M}_3 (t) & := \left\| \frac{\widehat{\kappa}}{ \mathcal{G} } \left( \frac{\partial \mathfrak{g}_{11}}{\partial X_2} - \frac{\partial \mathfrak{g}_{12}}{\partial X_1} \right) - \frac{\widehat{\kappa}}{2 \mathcal{G}^2} \left( \mathfrak{g}_{11} \frac{\partial \mathcal{G} }{\partial X_2} - \mathfrak{g}_{12} \frac{\partial \mathcal{G} }{\partial X_1} \right) \right\|_{L^\infty (U)} ,\\
\mathcal{M}_4 (t) &:= \left\| \frac{\mathfrak{g}_{22}}{ \mathcal{G}} \frac{\partial \widehat{\kappa}}{\partial X_1} -  \frac{\mathfrak{g}_{12}}{ \mathcal{G}} \frac{\partial \widehat{\kappa}}{\partial X_2} \right\|_{L^\infty (U)} + \left\| \frac{\mathfrak{g}_{11}}{ \mathcal{G}} \frac{\partial \widehat{\kappa}}{\partial X_1} -  \frac{\mathfrak{g}_{12}}{ \mathcal{G}} \frac{\partial \widehat{\kappa}}{\partial X_2} \right\|_{L^\infty (U)},\\
\mathcal{M}_5(t) &:= \left\| \frac{1}{ 2 \mathcal{G}} \frac{d \mathcal{G}}{d t} \right\|_{L^\infty (U)}, \text{ and }\mathcal{M}_{\mathfrak{j}} := \sup_{0 \leq t <T} \{ \mathcal{M}_{\mathfrak{j}} (t) \}{ \ }(\mathfrak{j}=1,2,3,4,5).
\end{align*}
Now we state the main results of this paper.
\begin{theorem}[Local existence $(\mathrm{I})$]\label{thm24}
Assume that $T < \infty$ and that
\begin{equation*}
 C_\sharp \mathcal{M}_1 ( C_A + 1 )  \leq \frac{1}{8 \sqrt{2}}.
\end{equation*}
Then for each $u_0 \in W_0^{1,2} ( \Gamma_0 )$, there exists a unique strong solution $u$ in
\begin{equation*}
L^2 (0,T_\star; W^{1,2}_0 \cap W^{2,2} ( \Gamma ( \cdot ) )) \cap W^{1,2} (0,T_\star ;L^2 (\Gamma ( \cdot )))
\end{equation*}
of system \eqref{eq11} with initial datum $u_0$, where
\begin{equation*}
T_\star = \min \left\{ T ,  \frac{1}{2} \log \left( 1 + \frac{1}{16  C_\star (C_A + 1)^2} \right) \right\}. 
\end{equation*}
Here $C_\sharp$, $C_A$, and $C_\star$ are the three positive constants in \eqref{eq15}, \eqref{eq16}, and \eqref{eq72}, respectively.
\end{theorem}
\noindent See Section \ref{sect4} for function spaces on the evolving surface $\Gamma (t)$.
\begin{theorem}[Local existence $(\mathrm{II})$]\label{thm25}
Assume that $T < \infty$ and that
\begin{equation*}
 C_\sharp ( \mathcal{M}_1 + \mathcal{M}_2 + \mathcal{M}_3 + \mathcal{M}_4 ) ( C_A + 1 )  \leq \frac{1}{8 \sqrt{2}}.
\end{equation*}
Then for each $u_0 \in W_0^{1,2} ( \Gamma_0 )$, there exists a unique strong solution $u$ in
\begin{equation*}
L^2 (0,T_*; W^{1,2}_0 \cap W^{2,2} ( \Gamma ( \cdot ) )) \cap W^{1,2} (0,T_*;L^2 (\Gamma ( \cdot )))
\end{equation*}
of system \eqref{eq11} with initial datum $u_0$, where
\begin{equation*}
T_* = \min \left\{ T, \frac{1}{2} \log \left( 1 + \frac{1}{16  C_\sharp^2 \mathcal{M}_5^2 (C_A + 1)^2} \right) \right\}. 
\end{equation*}
Here $C_\sharp$, $C_A$ are the two positive constants in \eqref{eq15} and \eqref{eq16}, respectively.
\end{theorem}
\begin{theorem}[Global existence and stability]\label{thm26}
Assume that $T = \infty$ and that
\begin{equation*}
C_\sharp ( \mathcal{M}_1 + \mathcal{M}_2 + \mathcal{M}_3 + \mathcal{M}_4 + \mathcal{M}_5 )  ( C_A + 1 ) \leq \frac{1}{8 \sqrt{2}}.
\end{equation*}
Then for each $u_0 \in W_0^{1,2} ( \Gamma_0 )$, there exists a unique strong solution $u$ in
\begin{equation*}
L^2 (0, \infty; W^{1,2}_0 \cap W^{2,2} ( \Gamma ( \cdot ) )) \cap W^{1,2} (0,\infty ;L^2 (\Gamma ( \cdot )))
\end{equation*}
of system \eqref{eq11} with initial datum $u_0$. Moreover, the solution $u$ satisfies the following three properties:\\
$(\mathrm{i})$ (\text{Energy equality}) For all $0 \leq s \leq t$,
\begin{equation*}
\frac{1}{2} \| u(t) \|_{L^2 (\Gamma (t))}^2 + \int_s^t \| \sqrt{\kappa} \nabla_\Gamma u (\tau ) \|^2_{L^2(\Gamma (\tau))} { \ } d \tau = \frac{1}{2} \| u(s) \|_{L^2 (\Gamma (s))}^2.
\end{equation*}
$(\mathrm{ii})$ (Stability) There is $C>0$ such that for all $t >0$
\begin{equation*}
\| u ( t) \|_{L^2 ( \Gamma (t))} \leq C t^{-\frac{1}{2}} \| \widehat{u}_0 \|_{W^{1,2} (U)}.
\end{equation*}
That is,
\begin{equation*}
\lim_{t \to \infty } \| u (t) \|_{L^2 ( \Gamma (t))} = 0.
\end{equation*}
$(\mathrm{iii})$ (Regularity) There is $C>0$ independent of $u_0$ such that
\begin{equation*}
\| D_t^w u \|_{L^2(0,T;L^2(\Gamma (\cdot )))} + \| {\rm{div}}_\Gamma (\kappa \nabla_\Gamma u ) \|_{L^2(0,T;L^2(\Gamma (\cdot )))} \leq C \| \widehat{u}_0 \|_{W^{1,2}(U)}.
\end{equation*}
Here $\widehat{u}_0 = \widehat{u}_0 (X) = u_0 (\widehat{x}(X,0))$ and $C_\sharp$, $C_A$ are the two positive constants in \eqref{eq15} and \eqref{eq16}.
\end{theorem}
\noindent Remark that we prove Theorems \ref{thm24}-\ref{thm26} in Sections \ref{sect5}-\ref{sect7}.

\section{Preliminaries}\label{sect3}
In this section, we first recall some basic properties of several differential operators on the evolving surface $\Gamma (t)$. Then we state the fundamental properties of an elliptic operator.

\subsection{Differential Operators on Evolving Surface $\Gamma (t)$}\label{subsec31}
We first introduce the representation of some differential operators. Then we state the surface divergence theorem. For $\psi = \psi (x)$, $\Psi = \Psi (x,t)$,
\begin{equation*}
\widehat{\psi} = \widehat{\psi} (X, t) := \psi (\widehat{x} (X,t)) \text{ and }\widehat{\Psi} = \widehat{\Psi} (X, t) := \Psi (\widehat{x} (X,t) , t).
\end{equation*}
From \cite[Chapter 3]{Jos11}, \cite[Appendix]{DE07}, \cite[Section 3]{K18}, and \cite[Section 3]{K19a}, we obtain the following lemma.
\begin{lemma}[Representation formula for differential operators]\label{lem31}{ \ }\\
$(\mathrm{i})$ For each $\psi \in C ( \mathbb{R}^3)$,
\begin{equation*}
\int_{\Gamma (t)} \psi (x) { \ } d \mathcal{H}^2_x = \int_U \widehat{\psi} (X,t) \sqrt{ \mathcal{G} (X,t) }{ \ }d X.
\end{equation*}
$(\mathrm{ii})$ For each $j=1,2,3$, and $\psi \in C^1 ( \mathbb{R}^3)$,
\begin{equation*}
\int_{\Gamma (t)} \partial^\Gamma_j \psi { \ } d \mathcal{H}^2_x = \int_U \mathfrak{g}^{\alpha \beta} \frac{\partial \widehat{x}_j }{\partial X_\alpha} \frac{\partial \widehat{\psi} }{\partial X_\beta} \sqrt{ \mathcal{G} }{ \ }d X.
\end{equation*}
$(\mathrm{iii})$ For each $i,j=1,2,3$, and $\psi \in C^2 ( \mathbb{R}^3)$,
\begin{equation*}
\int_{\Gamma (t)} \partial_i^\Gamma \partial_j^\Gamma \psi { \ }d \mathcal{H}^2_x =  \int_U \mathfrak{g}^{\zeta \eta} \frac{\partial \widehat{x}_i}{\partial X_\zeta} \frac{\partial}{\partial X_\eta} \left( \mathfrak{g}^{\alpha \beta} \frac{\partial \widehat{x}_j}{\partial X_\alpha} \frac{\partial \widehat{\psi}}{\partial X_\beta} \right) \sqrt{ \mathcal{G} } { \ }d X.
\end{equation*}
$(\mathrm{iv})$ For each $\psi \in C^2 ( \mathbb{R}^3)$,
\begin{equation*}
\int_{\Gamma (t)} {\rm{div}}_\Gamma ( \kappa \nabla_\Gamma \psi ) { \ } d \mathcal{H}^2_x = \int_U \left\{ \frac{1}{ \sqrt{ \mathcal{G}} } \frac{\partial}{\partial X_\alpha} \left( \widehat{\kappa} \sqrt{ \mathcal{G}} \mathfrak{g}^{ \alpha \beta} \frac{\partial \widehat{\psi}}{ \partial X_\beta} \right) \right\} \sqrt{ \mathcal{G} }{ \ }d X.
\end{equation*}
$(\mathrm{v})$
\begin{equation*}
\int_{\Gamma (t)} {\rm{div}}_\Gamma w { \ } d \mathcal{H}^2_x = \int_U \frac{1}{ 2 \mathcal{G} } \left( \frac{d \mathcal{G}}{d t} \right) \sqrt{ \mathcal{G} }{ \ }d X \left( = \frac{d}{d t} \int_U \sqrt{ \mathcal{G} }{ \ }d X \right).
\end{equation*}
$(\mathrm{vi})$ For every $\Psi  \in C^1 (\mathbb{R}^4)$,
\begin{equation*}
\int_{\Gamma (t)} D_t^w \Psi { \ } d \mathcal{H}^2_x = \int_U \left( \frac{d}{d t} \widehat{\Psi} \right) \sqrt{ \mathcal{G} }{ \ }d X.
\end{equation*}
\end{lemma}
\noindent Remark that one can derive system \eqref{eq12} from system \eqref{eq11} by applying Lemma \ref{lem31}. From \cite[Chapter 2]{Sim83} and \cite[Section 3]{K19b}, we obtain the following surface divergence theorem.
\begin{lemma}[Surface divergence theorem]\label{lem32}For every $\phi = { }^t ( \phi_1 , \phi_2 , \phi_3 ) \in [C^1 ( \mathbb{R}^3 )]^3$,
\begin{equation*}
\int_{\Gamma (t)} {\rm{div}}_\Gamma \phi { \ }d \mathcal{H}^2_x = - \int_{\Gamma (t)} H_\Gamma (n \cdot \phi) { \ } d \mathcal{H}^2_x + \int_{\partial \Gamma (t)} \nu \cdot \phi { \ } d \mathcal{H}^1_x,
\end{equation*}
where $H_\Gamma = H_\Gamma (x , t )$ denotes the mean curvature in the direction $n$ defined by $H_\Gamma = - {\rm{div}}_\Gamma n$. Here $n$ is the unit outer normal vector to $\overline{\Gamma (t)}$ and $\nu$ is the unit outer co-normal vector to $\partial \Gamma (t)$. 
\end{lemma}

\subsection{Elliptic Operators}\label{subsec32}
Let $1 < p < \infty$. Define
\begin{equation}\label{eq31}
\begin{cases}
A_p f & = - \left( \lambda_1 \frac{\partial^2}{\partial X_1^2} + \lambda_2 \frac{\partial^2 }{\partial X_2^2} \right) f,\\
D (A_p) & = L^p (U) \cap W^{1,p}_0 (U) \cap W^{2,p} (U).
\end{cases}
\end{equation}
We call $A_p$ the \emph{Dirichlet-Laplace operator} if $\lambda_1 = \lambda_2$. In particular, we write $A$ as $A_2$. We easily have the following basic properties of the operator $A_p$.
\begin{lemma}\label{lem33}
$(\mathrm{i})$ (Strictly ellipticity) For every $\xi = { }^t (\xi_1, \xi_2) \in \mathbb{R}^2$,
\begin{equation*}
\min \{ \lambda_1 , \lambda_2 \} | \xi |^2 \leq \lambda_1 \xi_1^2 + \lambda_2 \xi_2^2 \leq \max \{ \lambda_1 , \lambda_2 \} | \xi |^2.
\end{equation*}
$(\mathrm{ii})$ (Fundamental solution)
For $X = { }^t (X_1, X_2) \in \mathbb{R}^2 \setminus \{ { }^t (0,0) \}$,
\begin{equation*}
\mathcal{E} (X) := \frac{1}{2} \log ( \lambda_2 X_1^2 + \lambda_1 X_2^2).
\end{equation*}
Then
\begin{equation*}
- \lambda_1 \frac{\partial^2 \mathcal{E} }{\partial X_1^2} - \lambda_2 \frac{\partial^2 \mathcal{E} }{\partial X_2^2} = 0 \text{ in } \mathbb{R}^2 \setminus \{ { }^t (0,0) \} .
\end{equation*}
$(\mathrm{iii})$ (Formally selfadjoint operator) For all $f_\sharp , f_\flat \in D ( A)$,
\begin{equation*}
\int_U (A f_\sharp ) f_\flat { \ }d X = \int_U f_\sharp (A f_\flat ) { \ } d X.  
\end{equation*}
$(\mathrm{iv})$ (Fractional power of $A$) For all $f \in D (A)$
\begin{equation*}
\int_U (A f ) f { \ }d X = \lambda_1 \| \partial_{X_1} f \|_{L^2(U)}^2 + \lambda_2 \| \partial_{X_2} f \|_{L^2(U)}^2.  
\end{equation*}
$(\mathrm{v})$ (Heat system) Let $\phi = \phi (X,t)$ be a $C^2$-function. Set
\begin{equation*}
\Phi (X , t) = \phi ( X_1/\sqrt{\lambda_1} , X_2/\sqrt{\lambda_2} , t).
\end{equation*}
Assume that $\phi$ satisfies $\partial_t \phi - \Delta_X \phi = 0$. Then $\Phi$ satisfies
\begin{equation*}
\partial_t \Phi - (\lambda_1 \partial_{X_1}^2 + \lambda_2 \partial_{X_2}^2 ) \Phi = 0.
\end{equation*}
Here $\Delta_X := \partial_{X_1}^2 + \partial_{X_2}^2$.
\end{lemma}
From Lemma \ref{lem33}, \cite[Chapter 7]{Paz83}, \cite[Chapters 7-9]{GT98}, and \cite[Chapter 6]{Eva10}, we have the following lemma.
\begin{lemma}\label{lem34}
$(\mathrm{i})$ (Elliptic regularity) There is $C_\sharp = C_\sharp (p, \lambda_1 , \lambda_2 , U )>0$ such that for all $f \in D (A_p)$
\begin{equation}\label{eq32}
\| f \|_{L^p(U)} + \| \nabla_X f \|_{L^p(U)} + \| \nabla_X^2 f \|_{L^p (U)} \leq C_\sharp \| A_p f \|_{L^p(U)}.
\end{equation}
Moreover, if $\lambda_2 = q \lambda_1$ for some $q>0$, then we write
\begin{equation}\label{eq33}
C_\sharp = \frac{C_\sharp' ( q, p ,U )}{\lambda_1}.
\end{equation}
$(\mathrm{ii})$ (Interpolation inequality) For each $\varepsilon >0$ there is $C = C ( \varepsilon , p, \lambda_1 , \lambda_2 , U ) >0$ such that for all $f \in D (A_p)$
\begin{equation}\label{eq34}
\| \nabla_X f \|_{L^p (U)} \leq \varepsilon \| \nabla_X^2 f \|_{L^p(U)} + C \| f \|_{L^p(U)}.
\end{equation}
$(\mathrm{iii})$ The operator $-A_p$ generates a bounded analytic semigroup on $L^p(U)$.\\
$(\mathrm{iv})$ The operator $- A$ generates a contraction $C_0$-semigroup on $L^2(U)$.\\
$(\mathrm{v})$ The operator $A$ is a non-negative selfadjoint operator in $L^2(U)$.\\
$(\mathrm{vi})$ $D (A^{\frac{1}{2}}) = W^{1,2}_0 (U)$ and for all $f \in D (A^{\frac{1}{2}})$
\begin{equation}\label{eq35}
\| A^{\frac{1}{2}} f \|_{L^2 (U)}^2 = \lambda_1 \| \partial_{X_1} f \|_{L^2(U)}^2 + \lambda_2 \| \partial_{X_2} f \|_{L^2(U)}^2.  
\end{equation}
\end{lemma}
\begin{proof}[Proof of Lemma \ref{lem34}]
We only derive \eqref{eq33}. Assume that $\lambda_2 = q \lambda_1$ for some $q >0$. Set
\begin{equation}\label{eq36}
\begin{cases}
A_p' f = - \left( \frac{\partial^2}{\partial X_1^2} + q \frac{\partial^2 }{\partial X_2^2} \right) f,\\
D (A_p') = L^p(U) \cap W^{1,p}_0 (U) \cap W^{2,p}(U).
\end{cases}
\end{equation}
From the elliptic regularity \eqref{eq32} of $A_p'$, there is $C_\sharp' = C_\sharp' (q, p, U) >0$ such that for all $f \in D (A_p')$ 
\begin{equation*}
\| f \|_{L^p (U)} + \| \nabla_X f \|_{L^p(U)} + \| \nabla_X^2 f \|_{L^p(U)} \leq C_\sharp' \| A_p' f \|_{L^p(U)}.
\end{equation*}
Since $A_p = \lambda_1 A_p'$ and $D (A_p) = D (A_p')$, we have \eqref{eq33}. 
\end{proof}

Since $L^2(U)$ is a Hilbert space and $- A $ generates a bounded analytic semigroup on $L^2(U)$, it follows from \cite{Des64} to find that the operator $A$ has the maximal $L^p$-regularity (Proposition \ref{prop81}). Therefore we have the following lemma.
\begin{lemma}[Maximal $L^2$-regularity of $A$]\label{lem35}{ \ }\\
For each $T \in (0, \infty ]$ and $(F , V_0 ) \in L^2 (0, T ; L^2(U) ) \times D (A^{\frac{1}{2}})$, there exists a unique function $V$ satisfying the system: 
\begin{equation}\label{eq37}
\begin{cases}
\frac{d}{d t}V + A V = F & \text{ on } (0,T),\\
V|_{t =0} = V_0,
\end{cases}
\end{equation}
and the estimate:
\begin{multline}\label{eq38}
\| d V/{d t} \|_{L^2(0,T;H)} + \| A V \|_{L^2(0,T;L^2(U))} \\
\leq \sqrt{2} \| A^{\frac{1}{2}}V_0 \|_{L^2(U)} + C_A \| F \|_{L^2(0,T;L^2(U))} .
\end{multline}
Here the positive constant $C_A$ depends only on $(\lambda_1 , \lambda_2 , U)$. Moreover, if $\lambda_2 = q \lambda_1 $ for some $q >0$, then we write
\begin{equation}\label{eq39}
C_A = \frac{C_A' ( q ,U )}{\lambda_1}.
\end{equation}
\end{lemma}

\begin{proof}[Proof of Lemma \ref{lem35}]
Fix $V_0 \in D (A^{\frac{1}{2}})$ and $F \in L^2(0,\infty;L^2(U))$. We first show
\begin{equation}\label{Eq310}
\left( \int_0^\infty \| A \mathrm{e}^{- t A} V_0 \|_{L^2(U)}^2 { \ }d t \right)^{\frac{1}{2}} \leq \frac{1}{ \sqrt{2} } \| A^{\frac{1}{2}} V_0 \|_{L^2(U)}.
\end{equation}
Set $V_\natural (t) = \mathrm{e}^{- t A} A^{\frac{1}{2}} V_0$. Since $A$ is a non-negative selfadjoint operator and $\mathrm{e}^{- t A}$ is an analytic semigroup on $L^2(U)$, we check that
\begin{align*}
\frac{1}{2} \frac{d}{d t} \| V_\natural (t) \|_{L^2(U)}^2 = \int_U (- A V_\natural ) V_\natural { \ }d X = - \| A^{\frac{1}{2}} V_\natural (t) \|_{L^2 (U)}^2.
\end{align*}
Integrating with respect to time, we see that for all $t >0$
\begin{align*}
\frac{1}{2} \| V_\natural (t) \|_{L^2(U)}^2 + \int_0^t \| A^{\frac{1}{2}} V_\natural (t) \|_{L^2 (U)}^2 { \ }d t &= \frac{1}{2} \| V_\natural (0) \|_{L^2(U)}^2\\
 &= \frac{1}{2} \| A^{\frac{1}{2}} V_0 \|_{L^2(U)}^2.
\end{align*}
Since
\begin{equation*}
A^{\frac{1}{2}} V_\natural (t)  = A^{\frac{1}{2} } \mathrm{e}^{- t A} A^{\frac{1}{2}} V_0 = A \mathrm{e}^{- t A} V_0,
\end{equation*}
we have \eqref{Eq310}. 

To study system \eqref{eq37}, we now consider the following two systems:
\begin{equation}\label{Eq311}
\begin{cases}
\frac{d}{d t} V_\sharp + A V_\sharp = 0 \text{ on } (0,\infty),\\
V_\sharp |_{t =0} = V_0, 
\end{cases}
\end{equation}
\begin{equation}\label{Eq312}
\begin{cases}
\frac{d}{d t} V_\flat + A V_\flat = F \text{ on } (0,\infty),\\
V_\flat |_{t =0} = 0.
\end{cases}
\end{equation}
Since $- A$ generates an analytic $C_0$-semigroup on $L^2(U)$, we find that $V_\sharp = V_\sharp (t) =\mathrm{e}^{- t A} V_0$ and 
\begin{equation*}
V_\sharp \in C ([0, \infty ) ; L^2 (U)) \cap C((0, \infty); D (A)) \cap C^1 ((0,\infty ) ; L^2(U)).
\end{equation*}
Since
\begin{equation*}
\frac{d V_\sharp}{d t} = - A V_\sharp = - A \mathrm{e}^{ - t A} V_0,
\end{equation*}
it follows from \eqref{Eq310} to check that
\begin{equation}\label{Eq313}
\| d V_\sharp /{d t} \|_{L^2 (0, \infty ; L^2(U))} + \| A V_\sharp \|_{L^2 (0, \infty ; L^2(U))} \leq \sqrt{2} \| A^{\frac{1}{2}} V_0 \|_{L^2(U)}.
\end{equation}
From the maximal $L^2$-regularity of $A$ (Proposition \ref{prop81}), there exists a unique function $V_\flat$ satisfying system \eqref{Eq312} and
\begin{equation}\label{Eq314}
\| d V_\flat /{d t} \|_{L^2 (0, \infty ; L^2(U))} + \| A V_\flat \|_{L^2 (0, \infty ; L^2(U))} \leq C_A \| F \|_{L^2(0,\infty ; L^2(U))}.
\end{equation}
Here $C_A >0$ depends only on $U ,\lambda_1 , \lambda_2$. Set $V = V (t) = V_\natural (t) + V_\flat (t)$. It is easy to check that $V$ satisfies system \eqref{eq37}. From \eqref{Eq313} and \eqref{Eq314}, we have
\begin{multline*}
\| d V /{d t} \|_{L^2 (0, \infty ; L^2(U))} + \| A V \|_{L^2 (0, \infty ; L^2(U))} \\
\leq \sqrt{2} \| A^{\frac{1}{2}} V_0 \|_{L^2(U)} + C_A \| F \|_{L^2(0,\infty ; L^2(U))} .
\end{multline*}
Applying the maximal $L^2$-regularity of $A$, we easily deduce the uniqueness of solutions to system \eqref{eq37} with $(V_0 , F)$.

Finally, we derive \eqref{eq39}. To this end, we consider the following system:
\begin{equation}\label{Eq315}
\begin{cases}
\frac{d}{d t}W_\flat + A' W_\flat = \frac{F}{\lambda_1} \text{ on } (0,\infty),\\
W_\flat|_{t =0} = 0.
\end{cases}
\end{equation}
Here $A' := A_2'$ is the operator defined by \eqref{eq36}. From the maximal $L^2$-regularity of $A'$, there exists a unique function $W_\flat$ satisfying \eqref{Eq315} and
\begin{equation}\label{Eq316}
\| d W_\flat /{d t} \|_{L^2 (0, \infty ; L^2(U))} + \| A' W_\flat \|_{L^2 (0, \infty ; L^2(U))} \leq \frac{C_{A'}}{\lambda_1} \| F \|_{L^2(0,\infty ; L^2(U))} .
\end{equation}
Here $C_{A'} >0$ depends only on $q, U$. Now we set $V_{\flat \flat} = V_{\flat \flat} (t) = W_\flat (\lambda_1 t)$. It is easy to see that $V_{\flat \flat}$ satisfies \eqref{Eq312}. From \eqref{Eq316}, we have
\begin{equation}\label{Eq317}
\left\| \frac{d V_{\flat \flat}}{d t} \right\|_{L^2 (0, \infty ; L^2(U))} + \| A V_{\flat \flat} \|_{L^2 (0, \infty ; L^2(U))} \leq \frac{C_{A'}}{\lambda_1} \| F \|_{L^2(0,\infty ; L^2(U))} .
\end{equation}
Combing \eqref{Eq313} and \eqref{Eq317} gives
\begin{multline*}
\| d V /{d t} \|_{L^2 (0, \infty ; L^2(U))} + \| A V \|_{L^2 (0, \infty ; L^2(U))} \\
\leq \sqrt{2} \| A^{\frac{1}{2}} V_0 \|_{L^2(U)} + \frac{C_{A'}}{\lambda_1} \| F \|_{L^2(0,\infty ; L^2(U))} .
\end{multline*}
Thus, we have \eqref{eq39}. Therefore, the lemma follows. 
\end{proof}

\section{Function Spaces on Evolving Surfaces}\label{sect4}
In this section, we introduce and study function spaces on the evolving surface $\Gamma (t)$. Let $\widetilde{X} = \widetilde{X} (x,t)$ be the inverse mapping of $\widehat{x} = \widehat{x} (X,t)$. For each $0 \leq t <T$, $\psi = \psi (X)$, and $\varphi = \varphi (x,t)$,
\begin{align*}
\widehat{\psi} & = \widehat{\psi} (X , t) := \psi ( \widehat{x} (X, t)),\\
\widetilde{\varphi} &= \widetilde{\varphi}(x , t) : = \widetilde{\varphi} ( \widetilde{X}(x,t)).
\end{align*}
Moreover, for $\psi = \psi (x,t)$ and $\varphi (X,t)$,
\begin{align*}
\widehat{\psi} & = \widehat{\psi} (X , t) := \psi ( \widehat{x} (X, t) , t),\\
\widetilde{\varphi} &= \widetilde{\varphi}(x , t) : = \varphi ( \widetilde{X}(x,t) , t).
\end{align*}

\subsection{Function Spaces on Evolving Surface $\Gamma (t)$}\label{subsec41}
Let us define function spaces on the evolving surface $\Gamma (t)$. Throughout this subsection, we fix $t \in [0,T)$.

For $k=0,1,2$, and $1 \leq p < \infty$,
\begin{align*}
C^k ( \Gamma (t)) & := \{ \psi : \Gamma (t) \to \mathbb{R} ;{ \ }\psi = \widetilde{\varphi}, { \ }\varphi \in C^k ( U )\},\\
C^k ( \overline{\Gamma (t)}) & := \{ \psi : \overline{\Gamma (t)} \to \mathbb{R} ;{ \ }\psi = \widetilde{\varphi}, { \ }\varphi \in C^k (\overline{U})\},\\
C^k_0 ( \Gamma (t)) & := \{ \psi : \Gamma (t) \to \mathbb{R} ;{ \ }\psi = \widetilde{\varphi}, { \ }\varphi \in  C_0^k(U)\},\\
L^p ( \Gamma (t)) & := \{ \psi : \Gamma (t) \to \mathbb{R} ;{ \ }\psi = \widetilde{\varphi}, { \ }\varphi \in L^p(U),{ \ }\| \psi \|_{L^p (\Gamma (t))} < \infty \},\\
L^\infty ( \Gamma (t)) & := \{ \psi : \Gamma (t) \to \mathbb{R} ;{ \ }\psi = \widetilde{\varphi}, { \ }\varphi \in L^\infty (U),{ \ }\| \psi \|_{L^\infty (\Gamma (t))} < \infty \}.
\end{align*}
Here
\begin{align*}
\| \psi \|_{L^p( \Gamma (t))} & := \left( \int_{\Gamma (t)} | \widetilde{\varphi} (x,t) |^p { \ }d \mathcal{H}^2_x \right)^{\frac{1}{p}},\\
\| \psi \|_{L^\infty ( \Gamma (t))} & := ess.sup_{x \in \Gamma (t)}|  \widetilde{\varphi} (x,t) |.
\end{align*}
Moreover, for $\psi \in C (\overline{\Gamma (t)}) (= C^0(\overline{ \Gamma (t)}) )$,
\begin{equation*}
\| \psi |_{\partial \Gamma (t)} \|_{L^p( \partial \Gamma (t))} := \left( \int_{\partial \Gamma (t)} | { \ }\widetilde{\varphi} |_{\partial \Gamma (t)} { \ }|^p { \ }d \mathcal{H}^1_x \right)^{\frac{1}{p}}.
\end{equation*}
It is easy to check that
\begin{align}
\| \psi \|_{L^p (\Gamma (t))}^p & = \int_U | \widehat{ \widetilde{\varphi}} (X , t)|^p \sqrt{ \mathcal{G} (X,t)} { \ }d X\notag\\
& \leq C (\lambda_{max}) \int_U | \varphi |^p { \ }d X = C \| \varphi \|_{L^p(U)}^p\label{eq41}
\end{align}
and that
\begin{equation*}
\| \psi \|_{L^p (\Gamma (t))}^p =  \int_U | \varphi |^p \sqrt{ \mathcal{G} (X,t)} { \ }d X \geq \lambda_{min} \| \varphi \|_{L^p(U)}^p.
\end{equation*}
We also see that for all $\psi_\sharp \in L^p (\Gamma (t))$ and $\psi_\flat \in L^{p'}(\Gamma (t))$
\begin{equation}\label{eq42}
\left| \int_{\Gamma (t)} \psi_\sharp \psi_\flat { \ }d \mathcal{H}^2_x \right| \leq \| \psi_\sharp \|_{L^p(\Gamma (t))} \| \psi_\flat \|_{L^{p'}(\Gamma (t))},
\end{equation}
where $1 \leq p ,p' \leq \infty$ such that $1/p+1/{p'}=1$.

Next we define a weak derivative for functions on the surface $\Gamma (t)$. For $\psi \in C^k ( \overline{\Gamma (t)})$ or $\psi \in C_0^k ( \Gamma (t))$, we define the differential operators $\partial_j^\Gamma$ and $\partial_i^\Gamma \partial_j^\Gamma$ as in Lemma \ref{lem31}. 
\begin{definition}[Weak derivatives]\label{def41}
Let $1 \leq p \leq \infty$, $\psi \in L^p ( \Gamma (t) )$, and $i, j=1,2,3$.\\ 
$(\mathrm{i})$ We say that $\partial^\Gamma_j \psi \in L^p (\Gamma (t))$ if there exists $\Psi \in L^p ( \Gamma (t))$ such that for all $\phi \in C_0^1 ( \Gamma (t))$,
\begin{equation*}
\int_{\Gamma (t)} \Psi \phi { \ }d \mathcal{H}^2_x = - \int_{\Gamma (t)} \psi (\partial_j^\Gamma \phi + H_\Gamma n_j \phi ) { \ }d \mathcal{H}^2_x.
\end{equation*}
In particular, we write $\partial_j^\Gamma \psi$ as $\Psi$.\\
$(\mathrm{ii})$ We say that $\partial^\Gamma_j \partial_i^ \Gamma\psi \in L^p (\Gamma (t))$ if $\partial^\Gamma_i \psi \in L^p (\Gamma (t))$ and there exists $\Psi \in L^p ( \Gamma (t))$ such that for all $\phi \in C_0^1 ( \Gamma (t))$,
\begin{equation*}
\int_{\Gamma (t)} \Psi \phi { \ }d \mathcal{H}^2_x = - \int_{\Gamma (t)} \partial_i^\Gamma \psi (\partial_j^\Gamma \phi + H_\Gamma n_j \phi ) { \ }d \mathcal{H}^2_x.
\end{equation*}
In particular, we write $\partial_j^\Gamma \partial_i^\Gamma \psi$ as $\Psi$.
\end{definition}
\noindent We easily see the uniqueness of weak derivatives. See \cite{DEKR18} for weak derivatives for functions on a closed surface.

Now we introduce Sobolev spaces on the surface $\Gamma (t)$. For $1 \leq p < \infty$,
\begin{align*}
W^{1,p} ( \Gamma (t)) & := \{ \psi : \Gamma (t) \to \mathbb{R} ;{ \ }\psi = \widetilde{\varphi}, { \ }\varphi \in W^{1,p}(U),{ \ }\| \psi \|_{W^{1,p} (\Gamma (t))} < \infty \},\\
W_0^{1,p} ( \Gamma (t)) & := \{ \psi : \Gamma (t) \to \mathbb{R} ;{ \ }\psi = \widetilde{\varphi}, { \ }\varphi \in W_0^{1,p}(U),{ \ }\| \psi \|_{W^{1,p} (\Gamma (t))} < \infty \},\\
W^{2,p} ( \Gamma (t)) & := \{ \psi : \Gamma (t) \to \mathbb{R} ;{ \ }\psi = \widetilde{\varphi}, { \ }\varphi \in W^{2,p}(U),{ \ }\| \psi \|_{W^{2,p} (\Gamma (t))} < \infty \}.
\end{align*}
Here
\begin{align*}
\| \psi \|_{W^{1,p}( \Gamma (t))} & := \left( \int_{\Gamma (t)} ( | \widetilde{\varphi}(x,t) |^p +  | \nabla_\Gamma \widetilde{\varphi}(x,t)  |^p ) { \ }d \mathcal{H}^2_x \right)^{\frac{1}{p}},\\
\| \psi \|_{W^{2,p}( \Gamma (t))} & := \left( \int_{\Gamma (t)} ( | \widetilde{\varphi}(x,t) |^p +  | \nabla_\Gamma \widetilde{\varphi}(x,t)|^p +  | \nabla_\Gamma^2 \widetilde{\varphi}(x,t)  |^p ) { \ }d \mathcal{H}^2_x \right)^{\frac{1}{p}}.
\end{align*}
From Lemma \ref{lem31} we check that
\begin{multline}\label{eq43}
\int_{\Gamma (t)} | \partial_j^\Gamma \psi (x )|^p { \ }d \mathcal{H}^2_x = \int_U \left| \mathfrak{g}^{\alpha \beta} \frac{\partial \widehat{x}_j}{\partial X_\alpha} \frac{\partial \varphi}{\partial X_\beta} (X, t) \right|^p\sqrt{ \mathcal{G} (X,t)} { \ }d X\\
 \leq  C (\lambda_{min} , \lambda_{max})  \int_U | \nabla_X \varphi |^p { \ }d X \leq C \| \varphi \|_{W^{1,p}(U)}^p \text{ if } \psi \in C^1 ( \Gamma (t) ) 
\end{multline}
and that
\begin{equation}\label{eq44}
\int_{\Gamma (t)} | \partial_i^\Gamma \partial_j^\Gamma \psi (x )|^p { \ }d \mathcal{H}^2_x \leq C (\lambda_{min} , \lambda_{max} ) \| \varphi \|_{W^{2,p} (U)}^p \text{ if } \psi \in C^2 ( \Gamma (t)).
\end{equation}
From Lemmas \ref{lem42} and \ref{lem44}, we see that \eqref{eq43} holds for all $\psi \in W^{1,p} ( \Gamma (t))$ and that \eqref{eq44} holds for all $\psi \in W^{2,p} ( \Gamma (t))$.

\begin{lemma}[Properties of $W^{1,p}(\Gamma (t))$ and $W^{2,p} ( \Gamma (t))$]\label{lem42}{ \ }\\
$(\mathrm{i})$ Let $\psi \in W^{1,p} (\Gamma (t))$ and $\varphi \in W^{1,p}(U)$ such that $\psi = \widetilde{\varphi}$. Then for each $j=1,2,3$, $\partial_j^\Gamma \psi \in L^p (\Gamma (t))$ and
\begin{equation*}
\int_{\Gamma (t)} \partial_j^\Gamma \psi { \ }d \mathcal{H}^2_x = \int_U \mathfrak{g}^{\alpha \beta} \frac{\partial \widehat{x}_j}{\partial X_\alpha} \frac{\partial \varphi}{\partial X_\beta} \sqrt{ \mathcal{G} } { \ }d X.
\end{equation*}
$(\mathrm{ii})$ Let $\psi \in W^{2,p} (\Gamma (t))$ and $\varphi \in W^{2,p}(U)$ such that $\psi = \widetilde{\varphi}$. Then for each $i,j = 1,2,3$, $\partial_j^\Gamma \partial_i^\Gamma \psi \in L^p (\Gamma (t))$ and
\begin{equation*}
\int_{\Gamma (t)} \partial_i^\Gamma \partial_j^\Gamma \psi { \ }d \mathcal{H}^2_x =  \int_U \mathfrak{g}^{\zeta \eta} \frac{\partial \widehat{x}_i}{\partial X_\zeta} \frac{\partial}{\partial X_\eta} \left( \mathfrak{g}^{\alpha \beta} \frac{\partial \widehat{x}_j}{\partial X_\alpha} \frac{\partial \varphi }{\partial X_\beta} \right) \sqrt{ \mathcal{G} } { \ }d X.
\end{equation*}
$(\mathrm{iii})$ (Formula of the integration by parts) Let $1 < p,p' < \infty$ such that $1/p + 1/{p'}=1$. Then for each $j=1,2,3$, $\psi_\sharp \in W^{1,p} ( \Gamma (t))$, and $\psi_\flat \in W^{1,p'}_0 ( \Gamma (t))$,
\begin{equation}\label{eq45}
\int_{\Gamma (t)} \psi_\sharp ( \partial_j^\Gamma \psi_\flat ) { \ }d \mathcal{H}^2_x = - \int_{\Gamma (t)} (\partial_j^\Gamma \psi_\sharp + H_\Gamma n_j \psi_\sharp ) \psi_\flat { \ }d \mathcal{H}^2_x.
\end{equation}
\end{lemma}

\begin{proof}[Proof of Lemma \ref{lem42}]
We only prove $(\mathrm{i})$ since $(\mathrm{ii})$ and $(\mathrm{iii})$ are similar. Let $\psi \in W^{1,p} (\Gamma (t))$ and $\varphi \in W^{1,p}(U)$ such that $\psi = \widetilde{\varphi}$. Fix $j=1,2,3$. Since $C^1 (\overline{U}) \cap W^{1,p}(U)$ is dense in $W^{1,p} (U)$, there are $\varphi_m \in C^1( \overline{U}) \cap W^{1,p}(U)$ such that
\begin{equation}\label{eq46}
\lim_{m \to \infty} \| \varphi - \varphi_m \|_{W^{1,p}(U)} = 0.
\end{equation}
Set $\psi_m = \widetilde{\varphi_m}$. By definition, we find that $\psi_m \in C^1 (\overline{\Gamma (t) })$. Applying Lemma \ref{lem32}, we see that for all $\phi \in C_0^1 (\Gamma (t))$,
\begin{equation}\label{eq47}
\int_{\Gamma (t)} (\partial_j^\Gamma \psi_m ) \phi { \ }d \mathcal{H}^2_x = - \int_{\Gamma (t)} \psi_m (\partial_j^\Gamma \phi + H_\Gamma n_j \phi ) { \ }d \mathcal{H}^2_x.
\end{equation}
By \eqref{eq41}, \eqref{eq43}, and \eqref{eq46}, we see that
\begin{equation*}
\| \psi_m - \psi_{m'} \|_{W^{1,p} ( \Gamma (t))} \leq C \| \varphi_m - \varphi_{m'} \|_{W^{1,p} (U)} \to 0 { \ }(m,m' \to \infty ).
\end{equation*}
Since $L^p(\Gamma (t))$ is a Banach space, there is $\Psi \in L^p ( \Gamma (t))$ such that
\begin{equation}\label{eq48}
\lim_{m \to \infty} \| \partial_j^\Gamma \psi_m - \Psi \|_{L^p ( \Gamma (t))} =0.
\end{equation}
Using \eqref{eq42}, \eqref{eq47}, \eqref{eq48}, and
\begin{equation*}
\lim_{m \to \infty} \| \psi_m - \psi \|_{L^p ( \Gamma (t))} =0,
\end{equation*}
we check that for all $\phi \in C_0^1 ( \Gamma (t))$
\begin{equation*}
\int_{\Gamma (t)} \Psi \phi { \ }d \mathcal{H}^2_x = - \int_{\Gamma (t)} \psi (\partial_j^\Gamma \phi + H_\Gamma n_j \phi ) { \ }d \mathcal{H}^2_x.
\end{equation*}
This implies that
\begin{equation*}
\partial_j^\Gamma \psi \in L^p ( \Gamma (t)).
\end{equation*}
From the assertion $(\mathrm{ii})$ of Lemma \ref{lem31}, we have
\begin{equation*}
\int_{\Gamma (t)} \partial_j^\Gamma \psi_m{ \ }d \mathcal{H}^2_x = \int_U \mathfrak{g}^{\alpha \beta} \frac{\partial \widehat{x}_j}{\partial X_\alpha} \frac{\partial \varphi_m}{\partial X_\beta} \sqrt{ \mathcal{G} } { \ }d X.
\end{equation*}
Using \eqref{eq42}, the H\"{o}lder inequality, \eqref{eq46}, and \eqref{eq48}, we see that
\begin{equation*}
\int_{\Gamma (t)} \partial_j^\Gamma \psi { \ }d \mathcal{H}^2_x = \int_U \mathfrak{g}^{\alpha \beta} \frac{\partial \widehat{x}_j}{\partial X_\alpha} \frac{\partial \varphi}{\partial X_\beta} \sqrt{ \mathcal{G} } { \ }d X.
\end{equation*}
Note that $\int_{\Gamma (t)} 1 { \ } d \mathcal{H}^2_x + \int_U 1{ \ }d X < +\infty$. Therefore, the lemma follows. 
\end{proof}

Now we study the trace operator $\gamma: W^{1,p} (\Gamma (t)) \to L^p (\partial \Gamma (t))$.
\begin{proposition}\label{prop43}For each $\psi \in W_0^{1,p} (\Gamma (t))$,
\begin{equation*}
\| \gamma \psi \|_{L^p(\partial \Gamma (t))} = 0.
\end{equation*}
Here $\gamma : W^{1,p} (\Gamma (t)) \to L^p (\partial \Gamma (t))$ is the trace operator defined by the proof of Proposition \ref{prop43}.
\end{proposition}
To prove Proposition \ref{prop43}, we prepare the following lemma.
\begin{lemma}\label{lem44}Let $1 \leq p < \infty$. Then\\
$(\mathrm{i})$ $C_0(\Gamma (t))$ is dense in $L^p(\Gamma (t))$.\\
$(\mathrm{ii})$ $C_0^1(\Gamma (t))$ is dense in $W_0^{1,p}(\Gamma (t))$.\\
$(\mathrm{iii})$ $C^1 (\overline{ \Gamma (t) })$ is dense in $W^{1,p}(\Gamma (t))$.\\
$(\mathrm{iv})$ $C^2 (\overline{ \Gamma (t) })$ is dense in $W^{2,p}(\Gamma (t))$.
\end{lemma}
\begin{proof}[Proof of Lemma \ref{lem44}]
We only prove $(\mathrm{iii})$ since $(\mathrm{i})$, $(\mathrm{ii})$, and $(\mathrm{iv})$ are similar. Fix $\psi \in W^{1,p} ( \Gamma (t))$. By definition, there is $\varphi \in W^{1,p}(U)$ such that $\psi = \widetilde{\varphi}$. Since $C^1 (\overline{U}) \cap W^{1,p}(U)$ is dense in $W^{1,p} (U)$, there are $\varphi_m \in C^1( \overline{U}) \cap W^{1,p}(U)$ such that
\begin{equation}\label{eq49}
\lim_{m \to \infty} \| \varphi - \varphi_m \|_{W^{1,p}(U)} = 0.
\end{equation}
Set $\psi_m = \widetilde{\varphi_m}$. By definition, we find that $\psi_m \in C^1 (\overline{\Gamma (t) })$. By \eqref{eq41}, \eqref{eq43}, and \eqref{eq49}, we check that
\begin{align*}
\| \psi - \psi_m \|_{W^{1,p} (\Gamma (t))} & = \| \widetilde{\varphi} - \widetilde{\varphi_m} \|_{W^{1,p} (\Gamma (t))} \\
&\leq C \| \varphi - \varphi_m \|_{W^{1,p}(U)} \to 0 { \ }(m \to \infty ).
\end{align*}
Therefore, we conclude that $C^1 (\overline{ \Gamma (t) })$ is dense in $W^{1,p}(\Gamma (t))$. 
\end{proof}

\begin{proof}[Proof of Proposition \ref{prop43}]
We first introduce the trace operator $\gamma$ on $W^{1,p}( \Gamma (t))$. Let $\psi \in W^{1,p} (\Gamma (t))$. By definition, there are $\varphi \in W^{1,p}(U)$ and $\varphi_m \in C^1(\overline{U})$ such that $\psi = \widetilde{\varphi}$ and \eqref{eq49}. From the definition of line integral, we see that
\begin{equation*}
\int_{\partial \Gamma (t)} | f(x) |^p { \ }d \mathcal{H}^1_x = \int_{\partial U} |f (\widehat{x}(X,t))|^p | n_1^U \mathfrak{g}_1 - n_2^U \mathfrak{g}_2 |  { \ }d \mathcal{H}^1_X.
\end{equation*}
Here $n^U =n^U(X)= { }^t (n_1^U , n_2^U)$ is the unit outer normal vector at $ X \in \partial U$ and $\mathfrak{g}_\alpha = \partial \widehat{x}/{\partial X_\alpha}$. From $(n_1^U)^2 + (n_2^U) = 1$ and Definition \ref{def21}, we find that
\begin{equation*}
\| f \|_{L^p(\partial \Gamma (t))} \leq C(\lambda_{max}) \| \widehat{f} \|_{L^p(\partial U)}.
\end{equation*}
Since $\| \cdot \|_{L^p (\partial U)} \leq C (U) \| \cdot \|_{W^{1,p}(U)}$, we see that
\begin{align*}
\| \widetilde{\varphi_{m}} |_{\partial \Gamma (t)} - \widetilde{\varphi_{m'}} |_{\partial \Gamma (t)} \|_{L^p(\partial \Gamma (t))} & \leq C \| \varphi_{m} |_{\partial U} - \varphi_{m'} |_{\partial U} \|_{L^p(\partial U)}\\
 & \leq  C \| \varphi_{m} - \varphi_{m'} \|_{W^{1,p}(U)} \to 0 { \ }(m, m' \to \infty ).
\end{align*}
Therefore, we set
\begin{equation*}
\gamma \psi = \lim_{m \to \infty} \widetilde{\varphi}_m|_{\partial \Gamma (t)}\text{ in }L^p (\partial \Gamma (t)).
\end{equation*}
Since 
\begin{equation*}
\| \widetilde{\varphi}_m|_{\partial \Gamma (t)} \|_{L^p(\partial \Gamma (t))} \leq C \| \varphi_m |_{\partial U} \|_{L^p(\partial U)},
\end{equation*}
we use \eqref{eq49} to see that
\begin{equation*}
\| \gamma \psi \|_{L^p(\partial \Gamma (t))} \leq C \| \widehat{\gamma} \varphi \|_{L^p(\partial U)}.
\end{equation*}
Here $\widehat{\gamma}: W^{1,p}(U) \to L^p (\partial U)$ is the trace operator.

Now we assume that $\psi \in W^{1,p}_0 ( \Gamma (t))$. By definition, there are $\varphi \in W_0^{1,p}(U)$ and $\varphi_m \in C_0^1(U)$ such that $\psi = \widetilde{\varphi}$ and \eqref{eq49}. Since
\begin{equation*}
\| \widetilde{\varphi_m}|_{\partial \Gamma (t)} \|_{L^p(\partial \Gamma (t))} \leq C \| \varphi_m|_{\partial U} \|_{L^p(\partial U)} = 0,
\end{equation*}
we find that
\begin{equation*}
\| \gamma \psi \|_{L^p (\partial \Gamma (t))} = \lim_{m \to \infty} \| \widetilde{\varphi}_m|_{\partial \Gamma (t)} \|_{\partial \Gamma (t)} = 0.
\end{equation*}
Therefore we conclude that $\gamma \psi = 0$ if $\psi \in W^{1,p}_0 ( \Gamma (t) )$. 
\end{proof}

\subsection{Function Spaces on Evolving Surface $\Gamma_T$}\label{subsec42}

In this section we define and study function spaces on $\Gamma_T$. Set
\begin{align*}
&L^2(0,T; L^2( \Gamma ( \cdot )) ) = \{ \psi ; \psi = \widetilde{\varphi}, \varphi \in L^2(0,T;L^2(U)), \| \psi \|_{L^2(0,T;L^2( \Gamma (\cdot ) ))} < \infty \},\\
&L^2(0,T; W^{1,2} (\Gamma (\cdot )) ) \\
&{ \ \ \ \ }= \{ \psi ;{ \ }\psi = \widetilde{\varphi}, \varphi \in L^2 (0,T;W^{1,2}(U) ), \| \psi \|_{L^2(0,T;W^{1,2} (\Gamma (\cdot ) ) ) } < \infty \},\\
&L^2(0,T; W^{1,2}_0 \cap W^{2,2} (\Gamma (\cdot )) ) \\
&{ \ \ \ \ }= \{ \psi ;{ \ }\psi = \widetilde{\varphi}, \varphi \in L^2 (0,T;W^{1,2}_0 \cap W^{2,2} (U) ), \| \psi \|_{L^2(0,T;W^{2,2} (\Gamma (\cdot ) ) ) } < \infty \},\\
& W^{1,2}(0,T; L^2 (\Gamma (\cdot )) ) \\
&{ \ \ \ \ \ \ \ \ \ \ }= \{ \psi ;{ \ }\psi = \widetilde{\varphi}, \varphi \in W^{1,2} (0,T;L^2 (U) ), \| \psi \|_{W^{1,2}(0,T;L^2 (\Gamma (\cdot ) ) ) } < \infty \}.
\end{align*}
Here
\begin{align*}
\| \psi \|_{L^2(0,T;L^2(\Gamma (\cdot)) )} &:= \left( \int_0^T \| \widetilde{\varphi} \|_{L^2(\Gamma (t))}^2 { \ } d t \right)^{\frac{1}{2}},\\
\| \psi \|_{L^2(0,T;W^{1,2} (\Gamma (\cdot ) ) ) } & := \left( \int_0^T \| \widetilde{\varphi} \|_{W^{1,2}(\Gamma (t))}^2 { \ }d t \right)^{\frac{1}{2}},\\
\| \psi \|_{L^2(0,T;W^{2,2} (\Gamma (\cdot ) ) ) } & := \left( \int_0^T \| \widetilde{\varphi} \|_{W^{2,2}(\Gamma (t))}^2 { \ }d t \right)^{\frac{1}{2}},
\end{align*}
and
\begin{equation*}
\| \psi \|_{W^{1,2}(0,T;L^2 (\Gamma (\cdot ) ) ) } := \left( \int_0^T \| \widetilde{\varphi} \|_{L^2(\Gamma (t))}^2 { \ }d t + \int_0^T \left\| \frac{d \widetilde{\varphi}}{d t} \right\|_{L^2(\Gamma (t))}^2 { \ }d t \right)^{\frac{1}{2}}.
\end{equation*}
However, we can not define $\widetilde{\varphi}$ for $\varphi \in L^2 (0,T;L^2(U))$ in general. Therefore we define $\widetilde{\varphi}$ as follows: Let $\varphi \in L^2(0,T; L^2(U))$. From the definition of the Bochner integral and the Fubini-Tonelli theorem, there exists $\varphi_\sharp = \varphi_\sharp (X, t) \in L^2( U \times (0,T))$ such that
\begin{equation}\label{Eq410}
\| \varphi - \varphi_\sharp \|_{L^2(0,T;L^2(U))} = 0.
\end{equation}
Set
\begin{equation*}
\widetilde{\varphi} = \widetilde{\varphi}(x, t) : = \varphi_\sharp (\widetilde{X} (x,t) , t).
\end{equation*}
It is easy to check that
\begin{align*}
\| \psi \|_{L^2(0,T;L^2(\Gamma (\cdot )) )}^2 & = \int_0^T \| \widetilde{\varphi} (\cdot , t) \|_{L^2(\Gamma (t))}^2 { \ } d t\\
& = \int_0^T \int_U | \varphi_\sharp (X,t)|^2 \sqrt{\mathcal{G}(X,t)} { \ }d X d t\\
& \leq C \| \varphi_\sharp \|_{L^2(0,T;L^2(U))} = C \| \varphi \|_{L^2(0,T;L^2(U))}. 
\end{align*}
Now we study the case when $\varphi \in L^2 (0,T;W^{1,2}_0 \cap W^{2,2} (U) )$. Assume that $\varphi \in L^2 (0,T;W^{1,2}_0 \cap W^{2,2} (U) )$. We prove that for a.e. $0 \leq t <T$,
\begin{equation}\label{Eq411}
\varphi_\sharp (\cdot , t) \in W_0^{1,2}(U) \cap W^{2,2}(U).
\end{equation}
From \eqref{Eq410}, we see that for a.e. $0 \leq t <T$
\begin{equation*}
\| \varphi (t) - \varphi_\sharp (\cdot, t ) \|_{L^2(U)} = 0.
\end{equation*}
Fix $t$. It is easy to check that for all $\phi \in C_0^\infty (U)$ and $\alpha , \beta =1 ,2$,
\begin{align*}
\int_U \varphi_\sharp (X , t) \frac{\partial \phi}{\partial X_\alpha } { \ }d X &= \int_U \varphi (t) \frac{\partial \phi}{\partial X_\alpha } { \ }d X =- \int_U \frac{\partial \varphi}{\partial X_\alpha} \phi{ \ }d X,\\
\int_U \varphi_\sharp (X , t) \frac{\partial^2 \phi}{\partial X_\alpha \partial X_\beta} { \ }d X &= \int_U \varphi (t) \frac{\partial^2 \phi}{\partial X_\alpha \partial X_\beta} { \ }d X = \int_U \frac{\partial^2 \varphi}{\partial X_\alpha \partial X_\beta} \phi{ \ }d X.
\end{align*}
Since $\varphi (t) \in W^{2,2}(U)$, it follows from the definition of weak derivative for functions in $U$ to find that $\varphi_\sharp (\cdot , t) \in W^{2,2} (U)$. This implies that $\| \varphi_\sharp (\cdot, t) - \varphi (t)\|_{W^{2,2} (U)} =0$. Thus, we see \eqref{Eq411}. It is easy to check that
\begin{equation*}
\| \psi \|_{L^2(0,T;W^{2,2} (\Gamma (\cdot ) ) ) } \leq C \| \varphi \|_{L^2(0,T; W^{2,2} (U))}.
\end{equation*}

Finally, we define $\| \gamma \psi \|_{L^2 (0,T; L^2 (\partial \Gamma (\cdot)))}$ for $\psi \in L^2(0,T; W^{1,2} (\Gamma (\cdot)) )$. Let $\psi \in L^2(0,T; W^{1,2} (\Gamma (\cdot)) )$. By the previous argument, there are\\
 $\varphi \in L^2(0,T; W^{1,2} (U))$ and $\varphi_\sharp \in L^2( U \times (0,T)) $ such that \eqref{Eq410} and for a.e. $t$,
\begin{align*}
\varphi_\sharp ( \cdot , t ) \in W^{1,2}(U).
\end{align*}
Set
\begin{equation*}
\widetilde{\varphi} = \widetilde{\varphi}(x, t) : = \varphi_\sharp (\widetilde{X} (x,t) , t).
\end{equation*}
We define 
\begin{equation*}
\| \gamma \psi \|_{L^2 (0,T; L^2 (\partial \Gamma (\cdot)))} = \left( \int_0^T \| \gamma \widetilde{\varphi} \|_{L^2 (\partial \Gamma (t))}^2 { \ } d t \right)^{\frac{1}{2}} .
\end{equation*}
Here $\gamma$ is the trace operator defined by Proposition \ref{prop43}. By an argument similar to that in the proof of Proposition \ref{prop43}, we see that
\begin{align*}
\| \gamma \psi \|_{L^2 (0,T; L^2 (\partial \Gamma (\cdot)))}^2 & = \int_0^T \| \gamma \widetilde{\varphi} \|_{L^2 (\partial \Gamma (t))}^2 { \ } d t\\
& \leq C \int_0^T \| \widehat{\gamma } \varphi_\sharp \|_{L^2 (\partial U)}^2 { \ } d t\\
& \leq C(U) \int_0^T \|  \varphi_\sharp \|_{W^{1,2} (\partial U)}^2 { \ } d t = C \| \varphi \|_{L^2(0,T; W^{1,2} (U))}^2.
\end{align*}
 Remark that
\begin{align*}
\| \psi \|_{L^2(0,T;W^{1,2} (\Gamma (\cdot ) ) ) } & \leq C \| \varphi \|_{L^2(0,T; W^{1,2} (U))} \text{ if } \psi \in L^2(0,T;W^{1,2} (\Gamma (\cdot ) ) ),\\
\| \psi \|_{W^{1,2}(0,T;L^2 (\Gamma (\cdot ) ) ) } & \leq C \| \varphi \|_{W^{1,2} (0,T ; L^2(U))} \text{ if } \psi \in W^{1,2}(0,T;L^2 (\Gamma (\cdot ) ) ) .
\end{align*}

\section{On Strong Solutions to Advection-Diffusion Equation}\label{sect5}

In this section we study basic properties of the strong solutions to \eqref{eq11} with $u_0 \in W^{1,2}_0 ( \Gamma_0 )$. Let $L = L(t)$ be the evolution operator defined by \eqref{eq13} (see Section \ref{sect7} for details). For $\psi = \psi (x,t)$ and $\varphi (X,t)$, $\widehat{\psi} = \widehat{\psi} (X , t) := \psi ( \widehat{x} (X, t) , t)$ and $\widetilde{\varphi} = \widetilde{\varphi}(x , t) : = \varphi ( \widetilde{X}(x,t) , t)$.

\begin{lemma}[Strong solution to \eqref{eq11}]\label{lem51}
Let $u_0 \in W^{1,2}_0 ( \Gamma_0)$ and
\begin{equation*}
u \in L^2(0,T:W^{1,2}_0 \cap W^{2,2} (\Gamma (\cdot ))) \cap W^{1,2}(0,T;L^2(\Gamma (\cdot ))).
\end{equation*}
Assume that $\widehat{u} \in C ([0,T) ; L^2(U))$ and that the function $u$ satisfies the following three properties:\\
$(\mathrm{i})$ (Equation)
\begin{equation*}
\left\| \frac{d \widehat{u} }{d t} + L \widehat{u} \right\|_{L^2(0,T ; L^2(U))} = 0.
\end{equation*}
$(\mathrm{ii})$ (Boundary condition)
\begin{equation*}
\| \widehat{\gamma} \widehat{u} \|_{L^2(0,T;L^2 (\partial U))} = 0.
\end{equation*}
Here $\widehat{\gamma} : W^{1,2}(U) \to L^2( \partial U)$ is the trace operator.\\
$(\mathrm{iii})$ (Initial condition) 
\begin{equation*}
\lim_{t \to 0 + 0} \| \widehat{u} (t) - u_0 (\widehat{x} ( \cdot ,0)) \|_{L^2 (U)} = 0.
\end{equation*}
Then $u$ is a \emph{strong solution} to system \eqref{eq11} with initial datum $u_0$.
\end{lemma}

\begin{proof}[Proof of Lemma \ref{lem51}]From Lemma \ref{lem31} and an argument in Section \ref{sect4}, we see that
\begin{multline*}
\| D_t^w u + ({\rm{div}}_\Gamma w) u - {\rm{div}}_\Gamma (\kappa \nabla_\Gamma u) \|_{L^2(0,T;L^2(\Gamma (\cdot )))}^2 \\
= \int_0^T \left\| \left( \frac{d \widehat{u} }{d t} + L \widehat{u} \right) \sqrt{ \mathcal{G} } \right\|_{L^2(U)}^2{ \ }d t \leq C \int_0^T \left\| \frac{d \widehat{u} }{d t} + L \widehat{u}  \right\|_{L^2(U)}^2 { \ } d t = 0
\end{multline*}
and that
\begin{align*}
\| \gamma u \|_{L^2(0,T;L^2(\partial \Gamma (\cdot ) ))}^2 & = \int_0^T \| \gamma u \|_{L^2 (\partial \Gamma (t))}^2 { \ }d t\\
& \leq  C \int_0^T \| \widehat{\gamma} \widehat{u} \|_{L^2 (\partial U)}^2 { \ }d t = 0.
\end{align*}
Therefore, we find that $u$ satisfies the properties of strong solutions to system \eqref{eq11} with $u_0 \in W^{1,2}_0 ( \Gamma_0)$ as in Definition \ref{def22}. 
\end{proof}

\begin{lemma}[Sufficient condition for existence of a sol to \eqref{eq11}]\label{lem52}{ \ }\\
Let $v_0 \in W^{1,2}_0 (U)$ and
\begin{equation*}
v \in C ([0,T); L^2(U)) \cap L^2(0,T; W^{1,2}_0 \cap W^{2,2} (U)) \cap W^{1,2} (0,T ; L^2(U)).
\end{equation*}
Assume that $v$ satisfies
\begin{equation*}
\left\| \frac{d }{d t} v + L v  \right\|_{L^2(0,T;L^2(U))} = 0,
\end{equation*}
and
\begin{equation*}
\lim_{t \to 0 + 0} \| v (t) - v_0 \|_{L^2(U)} = 0.
\end{equation*}
Then there exists a strong solution $u$ of system \eqref{eq11} with $u_0$. Here $u_0 = u_0 (x) = v_0 (\widetilde{X} (x,0))$.
\end{lemma}

\begin{proof}[Proof of Lemma \ref{lem52}]
Since $v \in L^2(0,T ; L^2(U) )$, it follows from the definition of the Bochner integral and the Fubini-Tonelli theorem, there exists $v_\sharp = v_\sharp (X, t) \in L^2( U \times (0,T))$ such that
\begin{equation*}
\| v - v_\sharp \|_{L^2(0,T;L^2(U))} = 0.
\end{equation*}
By an argument similar to that in Section \ref{sect4}, we see that
\begin{equation*}
\| v - v_\sharp \|_{L^2(0,T;W^{2,2}(U))} = 0.
\end{equation*}
Now we set $u = u (x,t) = v_\sharp (\widetilde{X}(x,t) , t)$. It is clear that $\| \widehat{\gamma} \widehat{u} \|_{L^2(0,T;L^2 (\partial U))} = 0$ and that
\begin{equation*}
\| v - \widehat{u} \|_{L^2(0,T;W^{2,2}(U))} = 0.
\end{equation*}
By Lemma \ref{lem51} and the assumptions of Lemma \ref{lem52}, the lemma follows. 
\end{proof}

Next we study the basic properties of solutions to system \eqref{eq11}.
\begin{lemma}[Properties of solutions to \eqref{eq11}]\label{lem53}
Let $v_0 \in W^{1,2}_0 ( U )$ and
\begin{equation*}
v, v_\natural \in C ([0, T ); L^2(U)) \cap L^2(0, T; W^{1,2}_0 \cap W^{2,2} (U)) \cap W^{1,2} (0, T ; L^2(U)).
\end{equation*}
Set $u = \widetilde{v}$, $u_\natural = \widetilde{v_\natural}$, and $u_0 = u_0 (x) = v_0 (\widetilde{X} (x,0))$. Assume that $u$ and $u_\natural$ are two strong solutions of system \eqref{eq11} with initial datum $u_0$. Then\\
$(\mathrm{i})$ (Uniqueness)
\begin{equation*}
u = u_\natural \text{ on }[0, T ).
\end{equation*}
$(\mathrm{ii})$ (Energy equality) For all $0 \leq s \leq t < T$,
\begin{equation}\label{eq51}
\frac{1}{2} \| u(t) \|_{L^2 (\Gamma (t))}^2 + \int_s^t \| \sqrt{\kappa} \nabla_\Gamma u (\tau ) \|^2_{L^2(\Gamma (\tau))} { \ } d \tau = \frac{1}{2} \| u(s) \|_{L^2 (\Gamma (s))}^2.
\end{equation}
Moreover, assume in addition that $T = \infty$ and that there is $C >0$ independent of $v_0$ such that
\begin{equation}\label{eq52}
\left( \int_0^\infty \left\| \frac{d v}{d t} \right\|_{L^2(U)}^2 { \ }d t \right)^{\frac{1}{2}} + \left( \int_0^\infty \| A v \|_{L^2(U)}^2 { \ }d t \right)^{\frac{1}{2}} \leq C \| v_0 \|_{W^{1,2} (U)},
\end{equation}
where $A$ is the operator defined by \eqref{eq14}. Then,\\
$(\mathrm{iii})$ (Stability) There is $C>0$ such that for all $t >0$,
\begin{equation*}
\| u ( t) \|_{L^2 ( \Gamma (t))} \leq C t^{-\frac{1}{2}} \| v_0 \|_{W^{1,2} (U)}.
\end{equation*}
$(\mathrm{iv})$ (Regularity) There is $C>0$ independent of $v_0$ such that
\begin{equation*}
\| D_t^w u \|_{L^2 (0, \infty ; L^2 (\Gamma (\cdot ) ))} + \| {\rm{div}}_\Gamma ( \kappa \nabla_\Gamma u ) \|_{L^2 (0, \infty ; L^2 (\Gamma (\cdot ) ))} \leq C \| v_0 \|_{W^{1,2} (U)}.
\end{equation*}
\end{lemma}

\begin{proof}[Proof of Lemma \ref{lem53}]
We first show $(\mathrm{i})$ and $(\mathrm{ii})$. Set $u_\sharp = u- u_\natural$. Then 
\begin{equation*}
\begin{cases}
D_t^w u_\sharp + ({\rm{div}}_\Gamma w) u_\sharp - {\rm{div}}_\Gamma ( \kappa \nabla_\Gamma u_\sharp ) = 0 \text{ on } \Gamma_T ,\\
u_\sharp|_{\partial \Gamma_T} = 0,\\
u_\sharp |_{t = 0} = 0.
\end{cases}
\end{equation*}
Using Lemma \ref{lem31} and \eqref{eq45}, we see that
\begin{align*}
\frac{d}{d t} \int_{\Gamma (t)} \frac{1}{2} | u_\sharp |^2 { \ }d \mathcal{H}^2_x &= \int_{\Gamma (t)}\{  D_t^w u_\sharp + ({\rm{div}}_\Gamma w) u_\sharp \} u_\sharp { \ } d \mathcal{H}^2_x\\
& = \int_{\Gamma (t)}\{ {\rm{div}}_\Gamma ( \kappa \nabla_\Gamma u_\sharp )  \} u_\sharp { \ } d \mathcal{H}^2_x\\
& = - \int_{\Gamma (t)} \kappa | \nabla_\Gamma u_\sharp |^2 { \ } d \mathcal{H}^2_x.
\end{align*}
Integrating with respect to time, we check that for $0 < t <T$,
\begin{equation*}
\frac{1}{2} \| u_\sharp (t) \|_{L^2 ( \Gamma (t) ) }^2 + \int_0^t \| \sqrt{\kappa} \nabla_\Gamma u_\sharp \|^2_{L^2(\Gamma (\tau ))} { \ } d \tau = \frac{1}{2} \| u_\sharp (0) \|_{L^2 ( \Gamma_0 ) }^2 = 0.
\end{equation*}
Since $\kappa \geq \kappa_{min} >0$, we see that $u_\sharp = 0$ on $[0, T)$. Therefore, we conclude that $u = u_\natural$ on $[0, T )$. In the same manner, we have \eqref{eq51}.

Next we prove $(\mathrm{iii})$. From \eqref{eq51} and $\kappa \geq \kappa_{min} >0$, we check that for all $s < t$
\begin{equation*}
\| u (t) \|_{L^2 ({\Gamma (t)} )} \leq \| u(s) \|_{L^2 (\Gamma (s))}.
\end{equation*}
By the H\"{o}lder inequality, we see that
\begin{align*}
\| u (t) \|_{L^2 ({\Gamma (t)} )} & \leq \frac{1}{t} \int_0^t \| u(s) \|_{L^2 (\Gamma (s))} d s\\
& \leq \frac{1}{t^\frac{1}{2}} \left( \int_0^t \| u (s) \|_{L^2(\Gamma (s))}^2 { \ }d s \right)^{\frac{1}{2}} = : \text{(R.H.S.)}.
\end{align*}
Using \eqref{eq41}, \eqref{eq32}, and \eqref{eq52}, we observe that
\begin{align*}
\text{(R.H.S.)} & \leq \frac{C}{t^\frac{1}{2}} \left( \int_0^t \| v(s) \|_{L^2(U)}^2 { \ }d s \right)^{\frac{1}{2}}\\
& \leq \frac{C}{t^\frac{1}{2}} \left( \int_0^\infty \| A v (s) \|_{L^2(U)}^2 { \ }d s \right)^{\frac{1}{2}}\\
& \leq \frac{C}{t^\frac{1}{2}} \| v_0 \|_{W^{1,2} (U)}.
\end{align*}
Thus, we see $(\mathrm{iii})$. 

Finally, we prove $(\mathrm{iv})$. Using Lemma \ref{lem31} and \eqref{eq32}, we check that
\begin{multline*}
\left( \int_0^\infty \| D_t^w u \|_{L^2(\Gamma (t))}^2 { \ }d t \right)^{\frac{1}{2}} + \left( \int_0^\infty \| {\rm{div}}_\Gamma ( \kappa \nabla_\Gamma u ) \|_{L^2(\Gamma (t))}^2 { \ }d t \right)^{\frac{1}{2}}\\
\leq C \left( \int_0^\infty \left\| \frac{d v}{d t} \right\|_{L^2(U)}^2 { \ }d t \right)^{\frac{1}{2}} + C \left( \int_0^\infty \| A v \|_{L^2(U)}^2 { \ }d t \right)^{\frac{1}{2}} \leq C \| v_0 \|_{W^{1,2} (U)}.
\end{multline*}
Therefore, the lemma follows. 
\end{proof}

\section{Evolution Operator on Hilbert Space}\label{sect6}

In this section we provide the key method for constructing local and global-in-time strong solutions to the evolution equation \eqref{eq12}. To this end, we study an evolution operator on a Hilbert space by applying the maximal $L^p$-in-time regularity for Hilbert space-valued functions.

Let $H$ be a Hilbert space and $\| \cdot \|_H$ its norm. Let $T \in (0,\infty]$ and $\mathcal{A}: D (\mathcal{A}) (\subset H) \to H$ be a linear operator on $H$. For $0 \leq t <T $, let $\mathcal{B}(t) : D (\mathcal{B} (t)) \to H$ be a linear operator on $H$ such that $D (\mathcal{A}) \subset D (\mathcal{B} (t))$. Define $\mathcal{L}(t) f := \mathcal{A} f + \mathcal{B} (t) f $ and $D (\mathcal{L}(t)) := D (\mathcal{A})$. We consider the following evolution system:
\begin{equation}\label{eq61}
\begin{cases}
\frac{d}{d t}v + \mathcal{L} v = 0 & \text{ on } (0,T),\\
v|_{t =0} = v_0,
\end{cases}
\end{equation}
under the following assumptions:
\begin{assumption}\label{ass61}{ \ }\\
$(\mathrm{i})$ Assume that $- \mathcal{A}$ generates a bounded analytic semigroup on $H$.\\
$(\mathrm{ii})$ Assume that $- \mathcal{A}$ generates a contraction $C_0$-semigroup on $H$.\\
$(\mathrm{iii})$ The operator $\mathcal{A}$ is a non-negative selfadjoint operator on $H$.\\
$(\mathrm{iv})$ There are $C_1 , C_2 >0$ such that for all $f \in D (\mathcal{A})$ and $0 \leq t <T$, 
\begin{equation}\label{eq62}
\| \mathcal{B} (t) f \|_{H} \leq C_1 \| \mathcal{A} f \|_H + C_2 \| f \|_{H}.
\end{equation}
$(\mathrm{v})$ There is $C >0$ such that for all $\varphi \in C ((0,T) ; D (\mathcal{A}))$ and $0 < s \leq t <T$,
\begin{multline}\label{eq63}
\| \mathcal{B} (t) \varphi (t) - \mathcal{B} (s) \varphi (s) \|_H  \leq C (t - s) ( \| \mathcal{A} \varphi (t) \|_H + \| \varphi (t) \|_H ) \\
+ C ( \| \mathcal{A} \varphi (t) - \mathcal{A} \varphi (s) \|_H + \| \varphi (t) - \varphi (s) \|_H ).
\end{multline}
$(\mathrm{vi})$ For each fixed $t_0 \in [0,T )$ the operator $- \mathcal{L} (t_0)$ generates an analytic semigroup on $H$.
\end{assumption}
\noindent Since $-\mathcal{A}$ generates an analytic semigroup on $H$, it follows from the perturbation theory (Proposition \ref{prop82}) to see that $(\mathrm{vi})$ holds if $C_1$ is sufficiently small. Remark that $(\mathrm{vi})$ is not essential but important.

In this section we apply the maximal $L^2$-regularity of $\mathcal{A}$ to construct strong solutions to system \eqref{eq61}. 
\begin{lemma}[Maximal $L^2$-regularity of $\mathcal{A}$]\label{lem62}{ \ }\\
For each $T \in (0, \infty ]$ and $(F , V_0 ) \in L^2 (0, T ; H ) \times D ( \mathcal{A}^{\frac{1}{2}})$, there exists a unique function $V$ satisfying the following system:
\begin{equation*}
\begin{cases}
\frac{d}{d t}V + \mathcal{A} V = F & \text{ on } (0,T),\\
V|_{t =0} = V_0,
\end{cases}
\end{equation*}
and the estimate:
\begin{equation}\label{eq64}
\| d V/{d t} \|_{L^2(0,T;H)} + \| \mathcal{A} V \|_{L^2(0,T;H)} \leq \sqrt{2} \| \mathcal{A}^{\frac{1}{2}} V_0 \|_{H} + C_{\mathcal{A}} \| F \|_{L^2(0,T;H)}.
\end{equation}
Here the positive constant $C_{\mathcal{A}}$ does not depend on $T$, $F$, and $V_0$.
\end{lemma}
\noindent Since $- \mathcal{A}$ generates a bounded analytic semigroup on $H$ and $A$ is a non-negative selfadjoint operator on $H$, we use an argument similar to that in the proof of Lemma \ref{lem35} to have Lemma \ref{lem62}.

Before stating the main result of this section, we introduce the two function spaces and the definition of strong solutions to system \eqref{eq61}. Define
\begin{equation*}
Y_T = \{ \varphi \in C((0,T); H); { \ } \| \varphi \|_{Y_T} < \infty \}
\end{equation*}
with
\begin{equation*}
\| \varphi \|_{Y_T} := \sup_{0 < t <T}\{ \mathrm{e}^{- t} \| \varphi \|_H \} + \| d \varphi /{d t} \|_{L^2 (0,T;H)} + \| \mathcal{A} \varphi \|_{L^2 (0,T;H)}.
\end{equation*}
For each $0 < \mathfrak{a} \leq 1 $,
\begin{multline*}
C_{loc}^{\mathfrak{a}} ((0,T);H) := \{ \varphi \in C ((0,T) ;H); { \ } \text{for each }\varepsilon , T' >0 \text{ such that}\\
0< \varepsilon < T' < T \text{ there is }C (\varepsilon , T') >0\text{ such that}\\
\text{ for all } \varepsilon \leq s \leq t \leq T',{ \ }\| \varphi (t) - \varphi (s) \|_H \leq C(\varepsilon , T') (t-s)^{\mathfrak{a}} \}.
\end{multline*}
Let $v_0 \in D (\mathcal{A}^{\frac{1}{2}})$ and $v \in Y_T$. We call $v$ a \emph{strong solution} to system \eqref{eq61} with initial datum $v_0$ if the function $v$ satisfies the following two properties: $(\mathrm{i})$ $\| d v/{dt} + \mathcal{L} v \|_{L^2(0,T;H)}=0$, $(\mathrm{ii})$ $\lim_{t \to 0 + 0} v (t) = v_0$ in $H$.

Theorem \ref{thm63} is the main results of this section.
\begin{theorem}\label{thm63}
$(\mathrm{i})$ Assume that $T < \infty$ and that
\begin{equation}\label{eq65}
C_1 ( C_{\mathcal{A}} + 1 ) < \frac{1}{4 \sqrt{2}}.
\end{equation}
Then for each $v_0 \in D (\mathcal{A}^{\frac{1}{2}})$ system \eqref{eq61} admits a unique strong solution $v$ in $Y_{T_*}$, satisfying
\begin{equation*}
\| v \|_{Y_{T_*}} \leq 2 \| v_0 \|_H + 2 \sqrt{2} \| \mathcal{A}^{\frac{1}{2}} v_0 \|_H,
\end{equation*}
and for each $0 < t < T_*$
\begin{equation}\label{eq66}
v (t) = \mathrm{e}^{ - t \mathcal{A}} v_0 - \int_0^t \mathrm{e}^{- (t - \tau ) \mathcal{A}} \mathcal{B} v (\tau ) { \ }d \tau .
\end{equation}
Here
\begin{equation*}
T_* = \min \left\{ T , \frac{1}{2} \log \left( 1 + \frac{1}{16 C_2^2 (C_{\mathcal{A}} + 1)^2} \right) \right\}.
\end{equation*}
$(\mathrm{ii})$ Assume that $T = \infty$ and $C_2 = 0$. Suppose that \eqref{eq65} holds. Then for each $v_0 \in D ( \mathcal{A}^{\frac{1}{2}} )$ system \eqref{eq61} admits a unique strong solution $v$ in $Y_\infty$, satisfying
\begin{equation*}
\| v \|_{Y_\infty} \leq 2 \| v_0 \|_H + 2 \sqrt{2} \| \mathcal{A}^{\frac{1}{2}} V_0 \|_H,
\end{equation*}
and for each $0 < t < \infty$, \eqref{eq66} holds.\\
Here $C_1$, $C_2$, and $C_{\mathcal{A}}$ are the two positive constants appearing in \eqref{eq62} and the positive constant appearing in \eqref{eq64}, respectively.
\end{theorem}
To prove Theorem \ref{thm63}, we consider the following approximate equations.
\begin{proposition}\label{prop64}
Let $v_0 \in D ( \mathcal{A}^{\frac{1}{2}})$. For each $m \in \mathbb{N}$, set $v_1$ and $v_{m+1}$ as follows:
\begin{equation}\label{eq67}
\begin{cases}
\frac{d}{d t} v_1 + \mathcal{A} v_1 = 0 & \text{ on } (0,T),\\
v_1 |_{t =0} = v_0,
\end{cases}
\end{equation}
\begin{equation}\label{eq68}
\begin{cases}
\frac{d}{d t} v_{m+1} + \mathcal{A} v_{m+1} = -\mathcal{B} v_m & \text{ on } (0,T),\\
v_{m+1} |_{t =0} = v_0.
\end{cases}
\end{equation}
Then the following two assertions hold:\\
$(\mathrm{i})$ Assume that $T < \infty$. Then for each $m \in \mathbb{N}$,
\begin{align*}
\lim_{t \to 0 +0} v_m (t) & = v_0 \text{ in } H,\\
v_m & \in C ([0,T); H) \cap C ((0,T); D ( \mathcal{A})) \cap C^1 ((0,T);H),\\
v_m , \mathcal{A} v_m & \in C_{loc}^{\left(\frac{1}{2}\right)^m} ((0,T);H),\\
\| v_1 \|_{Y_T} & \leq \| v_0 \|_H + \sqrt{2} \| \mathcal{A}^{\frac{1}{2}} v_0 \|_H ,\\
\| v_{m+1} \|_{Y_T} \leq \| v_0 \|_H +&  \sqrt{2} \| \mathcal{A}^{\frac{1}{2}} v_0 \|_H + ( \sqrt{2} C_1 + C_2 \sqrt{ \mathrm{e}^{2 T} - 1 } )  (C_{\mathcal{A}} +1 ) \| v_m \|_{Y_T},\\
\| v_{m+2} - v_{m+1} \|_{Y_T} & \leq ( \sqrt{2} C_1 + C_2 \sqrt{ \mathrm{e}^{2 T} - 1 } ) ( C_{\mathcal{A}} + 1 ) \| v_{m+1} - v_m \|_{Y_T},
\end{align*}
and for each $0 < t <T$,
\begin{equation}\label{eq69}
v_{m+1} (t) = \mathrm{e}^{- t \mathcal{A}} v_0 - \int_0^t \mathrm{e}^{- (t-\tau ) \mathcal{A}} \mathcal{B}(\tau ) v_m (\tau ) { \ }d \tau .
\end{equation}
$(\mathrm{ii})$ Assume that $T = \infty$ and that $C_2 =0$. Then for each $m \in \mathbb{N}$,
\begin{align*}
\lim_{t \to 0 +0} v_m (t) & = v_0 \text{ in } H,\\
v_m & \in C ([0,\infty ); H) \cap C ((0, \infty); D ( \mathcal{A})) \cap C^1 ((0,\infty );H),\\
v_m , \mathcal{A} v_m & \in C_{loc}^{\left(\frac{1}{2}\right)^m} ((0, \infty );H),\\
\| v_1 \|_{Y_\infty} & \leq  \| v_0 \|_H + \sqrt{2} \| \mathcal{A}^{\frac{1}{2}} v_0 \|_H ,\\
\| v_{m+1} \|_{Y_\infty} & \leq \| v_0 \|_H + \sqrt{2} \| \mathcal{A}^{\frac{1}{2}} v_0 \|_H + \sqrt{2} C_1  (C_{\mathcal{A}} +1 ) \| v_m \|_{Y_\infty},\\
\| v_{m+2} - v_{m+1} \|_{Y_\infty} & \leq \sqrt{2} C_1 ( C_{\mathcal{A}} + 1 ) \| v_{m+1} - v_m \|_{Y_\infty},
\end{align*}
and for each $0 < t < \infty$, \eqref{eq69} holds.
\end{proposition}
To attack Proposition \ref{prop64}, we prepare the two lemmas.
\begin{lemma}\label{lem65}
Let $V_0 \in D ( \mathcal{A}^{\frac{1}{2}})$ and $F \in L^2 (0,T ; H)$. Let $V$ be the function satisfying system \eqref{eq62} with $(V_0, F)$, obtained by Lemma \ref{lem62}. Assume that 
\begin{equation*}
F \in C_{loc}^\mathfrak{a} ((0,T) ; H)
\end{equation*}
for some $0 < \mathfrak{a} \leq 1/2$. Then $\lim_{t \to 0 +0} V (t) = V_0$ in $H$,
\begin{equation}\label{Eq610}
V \in C([0,T) ; H) \cap C((0, T) ; D (\mathcal{A})) \cap C^1 ((0,T); H),
\end{equation}
and for $0 \leq t < T$,
\begin{equation*}
V (t) = \mathrm{e}^{- t \mathcal{A}} V_0 - \int_0^t \mathrm{e}^{- (t-\tau ) \mathcal{A}} F (\tau ) { \ }d \tau .
\end{equation*}
Moreover, 
\begin{equation}\label{Eq611}
V , \mathcal{A} V \in C_{loc}^{{\mathfrak{a}}/{2}} ((0,T) ; H).
\end{equation}
\end{lemma}

\begin{lemma}\label{lem66}
$(\mathrm{i})$ Assume $T < \infty$. Then for $\varphi \in C([0, T) ; H) \cap L^2 (0,T ; D (\mathcal{A}))$,
\begin{equation}\label{Eq612}
\| \mathcal{B} \varphi \|_{L^2(0,T;H)} \leq ( \sqrt{2} C_1 + C_2 \sqrt{ \mathrm{e}^{2 T} - 1 } ) \| \varphi \|_{Y_T}.
\end{equation}
$(\mathrm{ii})$ Assume $T = \infty$ and $C_2 = 0$. Then for $\varphi \in C([0, \infty) ; H) \cap L^2 (0,\infty ; D (\mathcal{A}))$,
\begin{equation*}
\| \mathcal{B} \varphi \|_{L^2(0, \infty ;H)} \leq \sqrt{2} C_1 \| \varphi \|_{Y_\infty}.
\end{equation*}
\end{lemma}

\begin{proof}[Proof of Lemma \ref{lem65}]
Fix $V_0 \in D ( \mathcal{A}^{\frac{1}{2}})$ and $F \in L^2(0,T;H) \cap C_{loc}^\mathfrak{a} ((0,T) ; H)$. Let $V$ be the function satisfying system \eqref{eq62} with $(V_0, F)$, obtained by Lemma \ref{lem62}. Since $- \mathcal{A}$ generates an analytic semigroup on $H$ and $F \in C_{loc}^\mathfrak{a} ((0,T) ; H)$, it follows from the semigroup theory to see that $\lim_{t \to 0 +0} V (t) = V_0$, \eqref{Eq610}, and that for $0 < s < t < T$,
\begin{align*}
V (t) & = \mathrm{e}^{- t \mathcal{A}} V_0 + \int_0^t \mathrm{e}^{ - (t - \tau ) \mathcal{A}} F ( \tau ) { \ }d \tau,\\
V (s) & = \mathrm{e}^{- s \mathcal{A}} V_0 + \int_0^s \mathrm{e}^{ - (s - \tau ) \mathcal{A} } F ( \tau ) { \ }d \tau.
\end{align*}
We now prove \eqref{Eq611}. Fix $\varepsilon , T'$ such that $0 < \varepsilon < T' <T$. Let $\varepsilon \leq s \leq t \leq T'$. We first consider the H\"{o}lder continuity of $V$. Since
\begin{multline*}
V (t) - V (s) = ( \mathrm{e}^{- (t -s ) \mathcal{A}} -1 ) \mathrm{e}^{- s \mathcal{A}} V_0\\
+ \int_s^t \mathrm{e}^{- (t - \tau ) \mathcal{A}} F ( \tau ) { \ }d \tau + \int_0^s ( \mathrm{e}^{- (t -s) \mathcal{A}} - 1) \mathrm{e}^{- (s - \tau ) \mathcal{A}} F (\tau ) { \ } d \tau,
\end{multline*}
we use Proposition \ref{prop83} and the H\"{o}lder inequality to see that
\begin{multline*}
\| V (t) - V (s) \|_H = C (t-s)^{\frac{1}{2}} \| \mathcal{A}^{\frac{1}{2}} \mathrm{e}^{- s \mathcal{A}} V_0 \|_H \\
+ C (t-s)^{\frac{1}{2}} \left( \int_s^t \| F ( \tau ) \|_H^2 { \ }d \tau \right)^{\frac{1}{2}} + C \int_0^s \frac{(t-s)^{\frac{1}{4}}}{(s - \tau )^{\frac{1}{4}}} \| F (\tau )\|_H { \ } d \tau\\
\leq C (t-s)^{\frac{1}{2}} \| \mathcal{A}^{\frac{1}{2}} V_0 \|_H + \{ C (t-s)^{\frac{1}{2}} + C (t-s)^{\frac{1}{4}} s^\frac{1}{4} \} \| F \|_{L^2(0,T;H)}\\
\leq C (T') (t - s )^{\frac{1}{4}}.
\end{multline*}
Next we consider the H\"{o}lder continuity of $\mathcal{A} V$. It is easy to check that
\begin{equation*}
\mathcal{A} V (t) - \mathcal{A} V (s) = H_1(t,s) + H_2(t,s) + H_3(t,s) + H_4 (t,s).
\end{equation*}
Here
\begin{align*}
H_1 &:= ( \mathrm{e}^{- (t -s ) \mathcal{A}} -1 ) \mathcal{A} \mathrm{e}^{- s \mathcal{A}} V_0,\\
H_2 &:= \mathcal{A} \int_s^t \mathrm{e}^{- (t - \tau ) \mathcal{A}} \{ F ( \tau ) - F (t) \} { \ }d \tau,\\
H_3 &:=\mathcal{A} \int_0^s \{ \mathrm{e}^{ - (t - s) \mathcal{A}} - 1 \} \mathrm{e}^{- (s - \tau ) \mathcal{A} }  \{ F (\tau ) - F (s) \} { \ }d \tau,\\
H_4 &:= \mathcal{A} \int_s^t \mathrm{e}^{- (t - \tau ) \mathcal{A}} F (t) { \ }d \tau + \mathcal{A} \int_0^s \{ \mathrm{e}^{ - (t - s) \mathcal{A}} - 1 \} \mathrm{e}^{- (s - \tau ) \mathcal{A} }  F (s  ) { \ }d \tau.
\end{align*}
From Proposition \ref{prop83} and $F \in C_{loc}^\mathfrak{a} ((0,T) ; H)$, we see that
\begin{equation*}
\| H_1 \|_H \leq C (t-s)^{\frac{1}{2}} \| \mathcal{A} \mathrm{e}^{- s \mathcal{A}} \mathcal{A}^{\frac{1}{2}} V_0 \|_H \leq  \frac{C (t-s)^{\frac{1}{2}}}{\varepsilon} \| \mathcal{A}^{\frac{1}{2}} V_0 \|_H
\end{equation*}
and that
\begin{equation*}
\| H_2 \|_{H} \leq \int_s^t C \frac{ (t - \tau)^\mathfrak{a}}{t - \tau} d \tau \leq C (t -s)^\mathfrak{a}.
\end{equation*}
We also see that
\begin{align*}
\| H_3 \|_{H} & \leq (t-s)^{\frac{\mathfrak{a}}{2}} C \left\| \int_0^s \mathcal{A}^{1 + \frac{\mathfrak{a}}{2}} \mathrm{e}^{- (s - \tau ) \mathcal{A} }  \{ F (\tau ) - F (s) \} { \ }d \tau \right\|_{H}\\
& \leq (t-s)^{\frac{\mathfrak{a}}{2}} C \int_0^s \frac{(s - \tau )^\mathfrak{a}}{(s - \tau )^{1 + \frac{\mathfrak{a}}{2}}} { \ }d \tau\\
& \leq C (t - s)^{\frac{\mathfrak{a}}{2}} s^\frac{\mathfrak{a}}{2} \leq C (t - s)^{\frac{\mathfrak{a}}{2}} {T'}^{\frac{\mathfrak{a}}{2}}.
\end{align*}
From $\frac{d}{d \tau} \mathrm{e}^{- \tau \mathcal{A}} f = - \mathcal{A} \mathrm{e}^{- \tau \mathcal{A}} f $, we find that
\begin{equation*}
H_4 =F (t) - F (s) -\mathrm{e}^{- (t -s )\mathcal{A}} \{ F (t) - F (s) \} + (\mathrm{e}^{- (t -s)\mathcal{A}} - 1) \mathrm{e}^{-s \mathcal{A}} F (s).
\end{equation*}
Since $F \in C_{loc}^\mathfrak{a} ((0,T) ; H)$ and
\begin{align*}
\| (\mathrm{e}^{- (t -s)\mathcal{A}} - 1) \mathrm{e}^{-s \mathcal{A}} F (s) \|_{H} & = C (t-s)^{\mathfrak{a}} \| \mathcal{A}^{\mathfrak{a}} \mathrm{e}^{- s \mathcal{A}} F (s) \|_{H}\\
& \leq \frac{C (t-s)^{\mathfrak{a}}}{\varepsilon^{\mathfrak{a}}} \sup_{\varepsilon \leq s \leq T'} \| F (s) \|_{H},
\end{align*}
we check that
\begin{equation*}
\| H_4 \|_{H} \leq C(\varepsilon , T') (t-s)^{\mathfrak{a}}.
\end{equation*}
As a result, we see that
\begin{equation*}
\| \mathcal{A} V (t) - \mathcal{A} V (s) \|_H = C( \varepsilon , T' ) (t - s )^{\frac{ \mathfrak{a} }{2}} .
\end{equation*}
Therefore, we conclude that
\begin{equation*}
\| V (t) - V (s) \|_H + \| \mathcal{A} V (t) - \mathcal{A} V (s) \|_H = C( \varepsilon , T' ) (t - s )^{\frac{ \mathfrak{a} }{2}}.
\end{equation*}
Since $\varepsilon , T'$ are arbitrary, we see \eqref{Eq611}. Therefore the lemma follows. 
\end{proof}

\begin{proof}[Proof of Lemma \ref{lem66}]
We only prove $(\mathrm{i})$. Assume $T < \infty$. Let $\varphi \in C([0, T) ; H) \cap L^2 (0,T ; D (\mathcal{A}))$. By \eqref{eq62}, we observe that
\begin{align*}
\int_0^T \| \mathcal{B} \varphi \|_H^2 { \ }d t \leq 2 C_1^2 \int_0^T \| \mathcal{A} \varphi \|_H^2 { \ }d t + 2 C_2^2 \int_0^T \mathrm{e}^{2 t} \mathrm{e}^{-2 t}\| \varphi \|_H^2 { \ }d t\\
\leq 2 C_1^2 \int_0^T \| \mathcal{A} \varphi \|_H^2 { \ }d t + C_2^2 (\mathrm{e}^{2 T} - 1) (\sup_{0 < t <T}\{ \mathrm{e}^{- t } \| \varphi \|_H \} )^2.
\end{align*}
Since $( a + b )^{\frac{1}{2}} \leq a^{\frac{1}{2}} + b^{\frac{1}{2}}$ $(a,b \geq 0)$, we check that 
\begin{align*}
\| \mathcal{B} \varphi \|_{L^2(0,T;H)} & \leq \sqrt{2} C_1 \| \mathcal{A} \varphi \|_{L^2 (0,T;H)} + C_2 \sqrt{ \mathrm{e}^{2 T} - 1 } \sup_{0 < t <T}\{ \mathrm{e}^{- t } \| \varphi \|_H \}\\
& \leq ( \sqrt{2} C_1 + C_2 \sqrt{ \mathrm{e}^{2 T} - 1 }) \| \varphi \|_{Y_T}.
\end{align*}
Thus, we have \eqref{Eq612}. Therefore the lemma follows. 
\end{proof}

Let us attack Proposition \ref{prop64}.
\begin{proof}[Proof of Proposition \ref{prop64}]
We only prove $(\mathrm{i})$ since $(\mathrm{ii})$ is similar. Assume that $T < \infty$. Fix $v_0 \in D (\mathcal{A}^{\frac{1}{2}})$ and $0 < \varepsilon <T' < T$. Let $s,t >0$ such that $\varepsilon \leq s \leq t \leq T'$.

We first consider $v_1$, i.e., system \eqref{eq67}. From the maximal $L^2$-regularity (Lemma \ref{lem62}) of $\mathcal{A}$, there is a unique function $v_1$ satisfying system \eqref{eq67} and
\begin{equation}\label{Eq613}
\| d v_1/d t \|_{L^2 (0,T; H)} + \| \mathcal{A} v_1 \|_{L^2(0,T; H)} \leq \sqrt{2} \| \mathcal{A}^{\frac{1}{2}} v_0 \|_H.
\end{equation}
Since $- \mathcal{A}$ generates an analytic $C_0$-semigroup on $H$, we see that
\begin{align*}
& \lim_{t \to 0+0} v_1 (t) = v_0 \text{ in } H,\\
& v_1 \in C ([0,T);H) \cap C ((0,T); D ( \mathcal{A})) \cap C^1 ((0,T);H),
\end{align*}
and for each $0 \leq t < T$
\begin{equation*}
v_1 (t) = \mathrm{e}^{- t \mathcal{A}} v_0.
\end{equation*}
Since $\mathrm{e}^{- t \mathcal{A}}$ is contraction $C_0$-semigroup on $H$, we check that
\begin{equation}\label{Eq614}
\sup_{0< t <T}\{ \mathrm{e}^{- t } \| v_1 \|_H \} = \sup_{0< t <T}\{ \mathrm{e}^{- t } \| \mathrm{e}^{- t \mathcal{A}} v_0 \|_H \} \leq \| v_0 \|_H.
\end{equation}
By \eqref{Eq613} and \eqref{Eq614}, we have
\begin{equation}\label{Eq615}
\| v_1 \|_{Y_T} \leq \| v_0 \|_H + \sqrt{2} \| \mathcal{A}^{\frac{1}{2}} v_0 \|_H.
\end{equation}
Next we consider the H\"{o}lder continuity of $v_1$ and $\mathcal{B} v_1$. From Proposition \ref{prop83}, we see that
\begin{align*}
\| v_1 (t) - v_1 (s) \|_H & = \| (\mathrm{e}^{- (t -s ) \mathcal{A}} - 1) \mathrm{e}^{- s \mathcal{A}} v_0 \|_H\\
& \leq C (t-s)^{\frac{1}{2}} \| \mathcal{A}^{\frac{1}{2}} \mathrm{e}^{- s \mathcal{A}} v_0 \|_H \leq C (t-s)^{\frac{1}{2}} \| \mathcal{A}^{\frac{1}{2}} v_0 \|_H
\end{align*}
and that
\begin{align*}
\| \mathcal{A} v_1 (t) - \mathcal{A} v_1 (s) \|_H & = \| (\mathrm{e}^{- (t -s ) \mathcal{A}} - 1) \mathcal{A} \mathrm{e}^{- s \mathcal{A}} v_0 \|_H\\
& \leq C (t-s)^{\frac{1}{2}} \| \mathcal{A} \mathrm{e}^{- s \mathcal{A}} \mathcal{A}^{\frac{1}{2}} v_0 \|_H \leq \frac{C (t-s)^{\frac{1}{2}}}{\varepsilon}\| \mathcal{A}^{\frac{1}{2}} v_0 \|_H .
\end{align*}
Therefore, we have
\begin{align}\label{Eq616}
\| v_1 (t) - v_1 (s) \|_H + \| \mathcal{A} v_1 (t) - \mathcal{A} v_1 (s) \|_H \leq C(\varepsilon ) (t-s)^{\frac{1}{2}} .
\end{align}
From \eqref{eq63} and \eqref{Eq616}, we observe that
\begin{align*}
\| \mathcal{B} (t) v_1 (t) - \mathcal{B} (s) v_1 (s) \|_H & \leq C (t - s) ( \| \mathcal{A} v_1 (t) \|_H + \| v_1 (t) \|_H ) + C(\varepsilon) (t-s)^{\frac{1}{2}}\\
& \leq C(\varepsilon , T') (t-s)^{\frac{1}{2}}.
\end{align*}
By \eqref{Eq612}, we find that
\begin{align}\label{Eq617}
\| \mathcal{B} v_1 \|_{L^2(0,T;H)} \leq ( \sqrt{2} C_1 + C_2 \sqrt{\mathrm{e}^{2 T} - 1} ) \| v_1 \|_{Y_T} < C (T) < + \infty .
\end{align}
Therefore, we conclude that
\begin{equation}\label{Eq618}
\mathcal{B} v_1 \in L^2(0,T;H) \text{ and } \mathcal{B} v_1 \in C_{loc}^{\frac{1}{2}} ((0,T);H).
\end{equation}

Next we consider $v_2$, i.e.,
\begin{equation}\label{Eq619}
\begin{cases}
\frac{d}{d t}v_2 + \mathcal{A} v_2 = - \mathcal{B} v_1 & \text{ on } (0,T),\\
v_2 |_{t=0} =v_0.
\end{cases}
\end{equation}
Since $v_0 \in D (\mathcal{A}^{\frac{1}{2}})$ and $\mathcal{B}v_1 \in L^2(0,T;H)$, it follows from Lemma \ref{lem62} to see that there exists a unique function $v_2$ satisfying system \eqref{Eq619} and 
\begin{multline}\label{Eq620}
\| d v_2/d t \|_{L^2 (0,T; H)} + \| \mathcal{A} v_2 \|_{L^2(0,T; H)}\\
\leq \sqrt{2} \| \mathcal{A}^{\frac{1}{2}} v_0 \|_H + C_{\mathcal{A}} \| \mathcal{B} v_1 \|_{L^2(0,T;H)} .
\end{multline}
Since $- \mathcal{A}$ generates an analytic semigroup on $H$ and \eqref{Eq618} holds, we see that
\begin{equation*}
v_2 \in C ( (0,T);H) \cap C ((0,T); D ( \mathcal{A})) \cap C^1 ((0,T);H)
\end{equation*}
and for each $0 < t < T$
\begin{equation*}
v_2 (t) = \mathrm{e}^{- t \mathcal{A}} v_0 - \int_0^t \mathrm{e}^{- (t - \tau ) \mathcal{A}} \mathcal{B} v_1 { \ } d \tau.
\end{equation*}
Since
\begin{equation*}
\| v_2 (t) \|_H \leq \| v_0 \|_H + t^{\frac{1}{2}} \left( \int_0^t \| \mathcal{B} v_1 \|_H^2 { \ }d \tau \right)^{\frac{1}{2}}
\end{equation*}
from $||| \mathrm{e}^{- t \mathcal{A}} ||| \leq 1$ and the H\"{o}lder inequality, we have
\begin{equation}\label{Eq621}
\sup_{0 < t < T} \{ \mathrm{e}^{- t } \| v_2 (t) \|_H \} \leq \| v_0 \|_H + \| \mathcal{B} v_1 \|_{L^2(0,T;H)}.
\end{equation}
In the same manner, we find that
\begin{equation*}
\lim_{t \to 0 + 0} \| v_2 (t ) - v_0 \|_H = 0.
\end{equation*}
By \eqref{Eq620}, \eqref{Eq621}, and \eqref{Eq617}, we check that
\begin{multline}\label{Eq622}
\| v_2 \|_{Y_T} \leq \| v_0 \|_H + \sqrt{2} \| \mathcal{A}^{\frac{1}{2}} v_0 \|_H + (C_{\mathcal{A}} + 1 )\| \mathcal{B} v_1 \|_{L^2(0,T;H)}\\
 \leq \| v_0 \|_H + \sqrt{2} \| \mathcal{A}^{\frac{1}{2}} v_0 \|_H + (C_{\mathcal{A}} + 1 ) ( \sqrt{2} C_1 + C_2 \sqrt{\mathrm{e}^{2 T} - 1} ) \| v_1 \|_{Y_T} .
\end{multline}
From \eqref{Eq618} and Lemma \ref{lem65}, we find that
\begin{equation}\label{Eq623}
v_2 , \mathcal{A} v_2 \in C_{loc}^{\frac{1}{4}} ((0,T);H).
\end{equation}
From \eqref{eq63} and \eqref{Eq623}, we see that
\begin{equation*}
\| \mathcal{B} (t) v_2 (t) - \mathcal{B} (s) v_2 (s) \|_H \leq C(\varepsilon , T') (t-s)^{\frac{1}{4}}.
\end{equation*}
By \eqref{Eq612}, \eqref{Eq615}, and \eqref{Eq622}, we observe that
\begin{equation*}
\| \mathcal{B} v_2 \|_{L^2(0,T;H)} \leq ( \sqrt{2} C_1 + C_2 \sqrt{ \mathrm{e}^{2 T} - 1 } ) \| v_2 \|_{Y_T} < C (T) < \infty .
\end{equation*}
Therefore, we conclude that
\begin{equation*}
\mathcal{B} v_2 \in L^2(0,T;H) \text{ and } \mathcal{B} v_2 \in C_{loc}^{\frac{1}{4}} ((0,T);H).
\end{equation*}

Now we assume that for each $m \in \mathbb{N}$ there are $C (m , \varepsilon, T') >0$ and $C(m ,T) >0$ such that
\begin{align}
\| v_{m} (t) - v_{m} (s) \|_H & \leq C (m ,\varepsilon , T') (t - s)^{ \left( \frac{1}{2} \right)^{m}}, \label{Eq624}\\
\| \mathcal{A} v_{m} (t) - \mathcal{A} v_{m} (s) \|_H & \leq C (m , \varepsilon , T') (t - s)^{ \left( \frac{1}{2} \right)^{m}} \label{Eq625},\\
\| v_{m} \|_{Y_T} & \leq C (m ,T) < \infty \label{Eq626},
\end{align}
and $\lim_{t \to 0+0} v_m (t) = v_0$ in $H$,
\begin{equation*}
v_{m} \in C ([0,T);H) \cap C ((0,T); D ( \mathcal{A})) \cap C^1 ((0,T);H).
\end{equation*}
From \eqref{eq63}, \eqref{Eq624}, and \eqref{Eq625}, we check that
\begin{multline*}
\| \mathcal{B} (t) v_{m} (t) - \mathcal{B} (s) v_{m} (s) \|_H \\
\leq C (t - s) ( \| \mathcal{A} v_{m} (t) \|_H + \| v_{m} (t) \|_H ) + C(m , \varepsilon , T')(t-s)^{\left( \frac{1}{2}\right)^{m}}\\
\leq C(m , \varepsilon , T') (t-s)^{\left( \frac{1}{2}\right)^{m}}.
\end{multline*}
From \eqref{Eq612} and \eqref{Eq626}, we find that
\begin{equation}\label{Eq627}
\| \mathcal{B} v_{m} \|_{L^2(0,T;H)} \leq ( \sqrt{2} C_1 + C_2 \sqrt{ \mathrm{e}^{2 T} - 1 } ) \| v_{m} \|_{Y_T} < +\infty.
\end{equation}
Therefore, we conclude that
\begin{equation}\label{Eq628}
\mathcal{B} v_{m} \in L^2(0,T;H) \text{ and } \mathcal{B} v_{m} \in C_{loc}^{\left( \frac{1}{2} \right)^{m}} ((0,T);H).
\end{equation}
Now we consider $v_{m+1}$, i.e.,
\begin{equation}\label{Eq629}
\begin{cases}
\frac{d}{d t}v_{m+1} + \mathcal{A} v_{m+1} = - \mathcal{B} v_{m} & \text{ on } (0,T),\\
v_{m+1} |_{t=0} =v_0.
\end{cases}
\end{equation}
Since $v_0 \in D (\mathcal{A}^{\frac{1}{2}})$ and \eqref{Eq628} holds, it follows from Lemma \ref{lem62} to see that there exists a unique function $v_{m+1}$ satisfying system \eqref{Eq629} and
\begin{multline}\label{Eq630}
\| d v_{m+1}/d t \|_{L^2 (0,T; H)} + \| \mathcal{A} v_{m+1} \|_{L^2(0,T; H)}\\
\leq \sqrt{2} \| \mathcal{A}^{\frac{1}{2}} v_0 \|_H + C_{\mathcal{A}} \| \mathcal{B} v_{m} \|_{L^2(0,T;H)}.
\end{multline}
Since $- \mathcal{A}$ generates an analytic semigroup on $H$, it follows from \eqref{Eq628} to find that
\begin{equation*}
v_{m+1} \in C ((0,T);H) \cap C ((0,T); D ( \mathcal{A})) \cap C^1 ((0,T);H)
\end{equation*}
and for each $0 < t < T$
\begin{equation*}
v_{m+1} (t) = \mathrm{e}^{- t \mathcal{A}} v_0 - \int_0^t \mathrm{e}^{- (t - \tau ) \mathcal{A}} \mathcal{B} v_{m} { \ } d \tau.
\end{equation*}
Since
\begin{equation*}
\| v_{m+1} (t) \|_H \leq \| v_0 \|_H + t^{\frac{1}{2}} \left( \int_0^t \| \mathcal{B} v_{m} \|_H^2 { \ }d \tau \right)^{\frac{1}{2}}
\end{equation*}
from $||| \mathrm{e}^{- t \mathcal{A}} ||| \leq 1$ and the H\"{o}lder inequality, we have
\begin{equation}\label{Eq631}
\sup_{0 < t < T} \{ \mathrm{e}^{- t } \| v_{m+1} (t) \|_H \} \leq \| v_0 \|_H + \| \mathcal{B} v_{m+1} \|_{L^2(0,T;H)}.
\end{equation}
It is easy to check that
\begin{equation*}
\lim_{t \to 0+0} \| v_{m+1} (t) - v_0 \|_H =0.
\end{equation*}
By \eqref{Eq630}, \eqref{Eq631}, and \eqref{Eq627}, we obtain
\begin{multline}\label{Eq632}
\| v_{m+1} \|_{Y_T} \leq \| v_0 \|_H + \sqrt{2} \| \mathcal{A}^{\frac{1}{2}} v_0 \|_H + (C_{\mathcal{A}} + 1 )\| \mathcal{B} v_{m} \|_{L^2(0,T;H)}\\
\leq \| v_0 \|_H + \sqrt{2} \| \mathcal{A}^{\frac{1}{2}} v_0 \|_H + (C_{\mathcal{A}} + 1 ) ( \sqrt{2} C_1 + C_2 \sqrt{ \mathrm{e}^{2 T} - 1 } ) \| v_{m} \|_{Y_T}.
\end{multline}
From Lemma \ref{lem65} and \eqref{Eq628}, we find that
\begin{equation}\label{Eq633}
v_{m+1} , \mathcal{A} v_{m+1} \in C_{loc}^{\left( \frac{1}{2} \right)^{m+1}} ((0,T);H).
\end{equation}
By \eqref{Eq632} and \eqref{Eq633}, we conclude that there are $C (m+1 ,\varepsilon , T') >0$ and $C (m+1 , T) >0$ such that 
\begin{align*}
\| v_{m+1} (t) - v_{m+1} (s) \|_H & \leq C (m+1 ,\varepsilon , T') (t - s)^{ \left( \frac{1}{2} \right)^{m+1}},\\
\| \mathcal{A} v_{m+1} (t) - \mathcal{A} v_{m+1} (s) \|_H & \leq C (m+1 , \varepsilon , T') (t - s)^{ \left( \frac{1}{2} \right)^{m+1}},\\
\| v_{m+1} \|_{Y_T} & \leq C (m + 1 ,T),
\end{align*}
and that $\lim_{t \to 0+0} v_{m+1} (t) = v_0$ in $H$,
\begin{equation*}
v_{m+1} \in C ([0,T);H) \cap C ((0,T); D ( \mathcal{A})) \cap C^1 ((0,T);H).
\end{equation*}
By induction, we see that for each $m \in \mathbb{N}$, $\lim_{t \to 0+0} v_{m} (t) = v_0$ in $H$,
\begin{align*}
v_m & \in C ([0,T); H) \cap C ((0,T); D ( \mathcal{A})) \cap C^1 ((0,T);H),\\
v_m , \mathcal{A} v_m & \in C_{loc}^{\left(\frac{1}{2}\right)^m} ((0,T);H),\\
\| v_1 \|_{Y_T} & \leq  \| v_0 \|_H + \sqrt{2} \| \mathcal{A}^{\frac{1}{2}} v_0 \|_H ,\\
\| v_{m+1} \|_{Y_T} & \leq \| v_0 \|_H + \sqrt{2} \| \mathcal{A}^{\frac{1}{2}} v_0 \|_H + ( \sqrt{2} C_1 + C_2 \sqrt{ \mathrm{e}^{2 T} - 1 } )  (C_{\mathcal{A}} +1 ) \| v_m \|_{Y_T},
\end{align*}
and for each $0 < t <T$,
\begin{equation*}
v_{m+1} (t) = \mathrm{e}^{- t \mathcal{A}} v_0 - \int_0^t \mathrm{e}^{- (t-\tau ) \mathcal{A}} \mathcal{B}(\tau ) v_m (\tau ) { \ }d \tau .
\end{equation*}
It remains to prove that
\begin{equation}\label{Eq634}
\| v_{m+2} - v_{m+1} \|_{Y_T} \leq ( \sqrt{2} C_1 + C_2 \sqrt{ \mathrm{e}^{2 T} - 1 } ) ( C_{\mathcal{A}} + 1 ) \| v_{m+1} - v_m \|_{Y_T}. 
\end{equation}
From
\begin{equation*}
\begin{cases}
\frac{d}{d t} v_{m+2} + \mathcal{A} v_{m+2} = - \mathcal{B} v_{m+1},\\
v_{m+2}|_{t=0} = v_0,
\end{cases}{ \ }\begin{cases}
\frac{d}{d t} v_{m+1} + \mathcal{A} v_{m+1} = - \mathcal{B} v_{m},\\
v_{m+1}|_{t=0} = v_0,
\end{cases}
\end{equation*}
we have
\begin{equation*}
\begin{cases}
\frac{d}{d t} (v_{m+2}- v_{m+1}) + \mathcal{A} (v_{m+2} - v_{m+1}) = - \mathcal{B} (v_{m+1} - v_m ) & \text{ on }(0,T),\\
(v_{m+2} - v_{m+1} )|_{t=0} = 0.
\end{cases}
\end{equation*}
Since $\mathcal{B} (v_{m+1} - v_m) \in L^2(0,T ;H)$, it follows from Lemma \ref{lem62} and \eqref{Eq612} to see that
\begin{multline}\label{Eq635}
\| d (v_{m+2} - v_{m+1})/d t \|_{L^2 (0,T;H)} + \| \mathcal{A} (v_{m+2} - v_{m+1} ) \|_{L^2 (0,T;H)} \\
\leq  C_{\mathcal{A}} \| \mathcal{B} (v_{m+2} - v_{m+1} ) \|_{L^2 (0,T;H)}\\
\leq ( \sqrt{2} C_1 + C_2 \sqrt{ \mathrm{e}^{2 T} - 1 }) C_{\mathcal{A} } \| v_{m+1} - v_m \|_{Y_T}. 
\end{multline}
From
\begin{equation*}
v_{m+2} (t) - v_{m+1} (t) = \int_0^t \mathrm{e}^{- (t - \tau ) \mathcal{A}} \mathcal{B} (v_{m+1} - v_m) { \ }d \tau,
\end{equation*}
we use \eqref{Eq612} and the H\"{o}lder inequality to check that
\begin{equation}\label{Eq636}
\sup_{0 < t <T} \{ \mathrm{e}^{-t} \| v_{m+1} - v_m \|_H \} \leq (\sqrt{2} C_1 + C_2 \sqrt{ \mathrm{e}^{2 T} - 1 }) \| v_{m+1} - v_m \|_{Y_T}.
\end{equation}
Combining \eqref{Eq635} and \eqref{Eq636} gives \eqref{Eq634}. Therefore the proof of Proposition \ref{prop64} is finished. 
\end{proof}

Finally, we prove Theorem \ref{thm63}.
\begin{proof}[Proof of Theorem \ref{thm63}]
Fix $v_0 \in D ( \mathcal{A}^{\frac{1}{2}})$. Let $v_1$ and $v_{m+1}$ be the solutions to system \eqref{eq67} and \eqref{eq68}, obtained by Proposition \ref{prop64}. We first show $(\mathrm{i})$. Assume that $T < \infty$ and that
\begin{equation*}
C_1 ( C_{\mathcal{A}} + 1 ) < \frac{1}{4 \sqrt{2}}.
\end{equation*}
Write
\begin{equation*}
T_* = \min \left\{ T, \frac{1}{2} \log \left( 1 + \frac{1}{16 C_2^2 (C_{\mathcal{A}} + 1)^2} \right) \right\}.
\end{equation*}
It is easy to check that
\begin{align}
\sqrt{2} C_1 (C_{\mathcal{A}} + 1) & \leq \frac{1}{4},\label{Eq637}\\
C_2 ( C_{\mathcal{A}} + 1) \sqrt{ \mathrm{e}^{2 T_* } - 1} & \leq \frac{1}{4}.\label{Eq638}
\end{align}
From the assertion $(\mathrm{i})$ of Proposition \ref{prop64}, \eqref{Eq637}, and \eqref{Eq638}, we find that
\begin{align*}
\| v_{m+1} \|_{Y_{T_*}} & \leq \| v_0 \|_H + \sqrt{2} \| \mathcal{A}^{\frac{1}{2}} v_0 \|_H + \frac{1}{2} \| v_m \|_{Y_{T_*}},\\
\| v_{m+2} - v_{m+1} \|_{Y_{T_*}} & \leq \frac{1}{2} \| v_{m+1} - v_m \|_{Y_{T_*}}. 
\end{align*}
Since $\| v_1 \|_{Y_{T_*}} \leq \| v_0 \|_H + \sqrt{2} \| \mathcal{A}^{\frac{1}{2}} v_0 \|_H$, we see that for each $m \in \mathbb{N}$,
\begin{equation*}
\| v_m \|_{Y_{T_*}} \leq  2 \| v_0 \|_H + 2 \sqrt{2} \| \mathcal{A}^{\frac{1}{2}} v_0 \|_H ,
\end{equation*}
and
\begin{align*}
\| v_{m+2} - v_{m+1} \|_{Y_{T_*}} & \leq \left( \frac{1}{2} \right)^m \| v_2 - v_1 \|_{Y_{T_*}} \\
& \leq \left( \frac{1}{2} \right)^{m} \{ 3 \| v_0 \|_H + 3 \sqrt{2} \| \mathcal{A}^{\frac{1}{2}} v_0 \|_H \}. 
\end{align*}
From a fixed point argument, we have a unique function $v$ in $Y_{T_*}$ satisfying 
\begin{align}\label{Eq639}
& \lim_{m \to \infty} \| v_m - v \|_{Y_{T_*}} = 0,\\
& \| v \|_{Y_{T_*}} \leq 2 \| v_0 \|_H + 2 \sqrt{2} \| \mathcal{A}^{\frac{1}{2}} v_0 \|_H .\notag
\end{align}
Using \eqref{Eq612}, \eqref{Eq637}, \eqref{Eq638}, and \eqref{Eq639}, we see that
\begin{equation}\label{Eq640}
\mathcal{B} v \in L^2 (0,T;H).
\end{equation}
Since $v_{m+1}$ satisfies the system
\begin{equation*}
\begin{cases}
\frac{d}{d t} v_{m+1} + \mathcal{A} v_{m+1} = - \mathcal{B} v_m \text{ on } (0,T_*),\\
v_{m+1} |_{t = 0} = v_0,
\end{cases}
\end{equation*}
and that for $0 < t < T_*$
\begin{equation*}
v_{m+1} (t) = \mathrm{e}^{- t \mathcal{A}} v_0 - \int_0^t \mathrm{e}^{- (t - \tau ) \mathcal{A}} \mathcal{B} v_m { \ }d \tau,
\end{equation*}
we apply \eqref{Eq639} and \eqref{Eq612} to see that
\begin{equation*}
\left\| \frac{d}{d t} v + \mathcal{A} v + \mathcal{B} v \right\|_{L^2(0,T;H)} = 0,
\end{equation*}
and that for $0 < t < T_*$
\begin{equation*}
v (t) = \mathrm{e}^{- t \mathcal{A}} v_0 - \int_0^t \mathrm{e}^{- (t - \tau ) \mathcal{A}} \mathcal{B} v { \ }d \tau.
\end{equation*}
Since $\mathrm{e}^{- t \mathcal{A}}$ is a contraction $C_0$-semigroup on $H$ and \eqref{Eq640} holds, we use the H\"{o}lder inequality to check that
\begin{align*}
\| v (t) - v_0 \|_H & \leq \| \mathrm{e}^{- t \mathcal{A}} v_0 - v_0 \|_H + t^{\frac{1}{2}} \| \mathcal{B} v \|_{L^2(0,T;H)}\\
& \to 0 { \ }( t \to 0 + 0).
\end{align*}
Therefore the assertion $(\mathrm{i})$ of Theorem \ref{thm63} is proved. The assertion $(\mathrm{ii})$ is similar. 
\end{proof}

\section{Existence of Strong Solutions to Advection-Diffusion Equation}\label{sect7}
In this section we prove the existence of local and global-in-time strong solution to system \eqref{eq11}. From Lemmas \ref{lem51} and \ref{lem52}, we only have to consider the existence of strong solutions to system \eqref{eq12}. To consider \eqref{eq12}, we write system \eqref{eq12} as follows:
\begin{equation}\label{eq71}
\begin{cases}
\frac{d}{d t} v + A v = - B v & \text{ on } (0,T),\\
v|_{t = 0} = v_0,
\end{cases}
\end{equation}
where $B = B (t) = L (t) - A$. Here $L$ and $A(=A_2)$ are the two operators defined by \eqref{eq13} and \eqref{eq31}, respectively. Set
\begin{equation*}
\begin{cases}
B f = B (t) f & = (L - A )f ,\\
D (B (t)) & = D (A),
\end{cases}{ \ }\begin{cases}
L f  & = L (t) f,\\
D (L (t)) & = D (A).
\end{cases}
\end{equation*}
In this section, we apply Theorem \ref{thm63} to construct strong solutions to system \eqref{eq71}. Set $\mathcal{M}_{\mathfrak{j}} (t)$ and $\mathcal{M}_{\mathfrak{j}}$ as in Section \ref{sect2} $(\mathfrak{j} = 1,2,3,4,5)$.

\subsection{Time-dependent Laplace-Beltrami Operator}\label{subsec71}
The aim of this subsection is to prove the following key lemma to apply Theorem \ref{thm63}.
\begin{lemma}\label{lem71}Let $C_\sharp$ be the positive constant appearing in \eqref{eq32}. Then\\
$(\mathrm{i})$ There is $C_\star >0$ such that for all $f \in D (A)$ and $0 \leq t <T$, 
\begin{equation}\label{eq72}
\| B (t) f \|_{L^2 (U)} \leq 2 C_\sharp \mathcal{M}_1 \| A f \|_{L^2 (U)} + C_\star \| f \|_{L^2(U)}.
\end{equation}
$(\mathrm{ii})$ For all $f \in D (A)$ and $0 \leq t <T$, 
\begin{equation}\label{eq73}
\| B (t) f \|_{L^2(U)} \leq 2 C_\sharp \sum_{\mathfrak{j}=1}^4 \mathcal{M}_{\mathfrak{j}} \| A f \|_{L^2(U)} + \mathcal{M}_5 \| f \|_{L^2(U)}.
\end{equation}
$(\mathrm{iii})$ For all $f \in D (A)$ and $0 \leq t <T$, 
\begin{equation}\label{eq74}
\| B (t) f \|_{L^2(U)} \leq 2 C_\sharp \left( \mathcal{M}_1 + \mathcal{M}_2 + \mathcal{M}_3 + \mathcal{M}_4 + \mathcal{M}_5 \right) \| A f \|_{L^2(U)}.
\end{equation}
$(\mathrm{iv})$ There is $C >0$ such that for all $\varphi \in C ((0,T) ; D (A) )$ and $0 < s \leq t <T$,
\begin{multline}\label{eq75}
\| B (t) \varphi (t) - B (s) \varphi (s) \|_{L^2(U)}  \leq C (t - s) ( \| A \varphi (t) \|_{L^2(U)} + \| \varphi (t) \|_{L^2(U)} ) \\
+ C ( \| A \varphi (t) - A \varphi (s) \|_{L^2(U)} + \| \varphi (t) - \varphi (s) \|_{L^2(U)} ).
\end{multline}
\end{lemma}
Applying Lemmas \ref{lem34}, \ref{lem71}, and Proposition \ref{prop82}, we have the following lemma.
\begin{lemma}\label{lem72}
Assume that
\begin{equation*}
2 C_\sharp \mathcal{M}_1 \leq \frac{1}{4}.
\end{equation*}
Then for each fixed $t_0\in [0, T)$ the operator $- L (t_0)$ generates an analytic semigroup on $L^2(U)$.
\end{lemma}

\begin{proof}[Proof of Lemma \ref{lem71}]
Fix $f \in D (A)$. A direct calculation shows that
\begin{multline*}
B (t) = L (t) f - A f \\
=  - \frac{1}{\sqrt{ \mathcal{G}}} \frac{\partial}{\partial X_\alpha} \left( \widehat{\kappa} \sqrt{ \mathcal{G}} \mathfrak{g}^{\alpha \beta } \frac{\partial f }{\partial X_\beta} \right) + \frac{1}{2 \mathcal{G}} \left( \frac{d \mathcal{G}}{d t} \right) f + \left(\lambda_1 \frac{\partial^2 f}{\partial X_1^2} + \lambda_2 \frac{\partial^2 f}{\partial X_2^2} \right) f\\
= B_1 (t) f + B_2 (t) f  + B_3 (t) f  + B_4 (t) f + B_5(t) f.
\end{multline*}
Here
\begin{align*}
B_1 (t) f & := - ( \widehat{\kappa} \mathfrak{g}^{11} - \lambda_1 ) \frac{\partial^2 f }{\partial X_1^2} - ( \widehat{\kappa} \mathfrak{g}^{2 2} - \lambda_2 ) \frac{\partial^2 f }{\partial X_2^2} - 2 \widehat{\kappa} \mathfrak{g}^{1 2} \frac{\partial^2 f }{\partial X_1 \partial X_2},\\
B_2 (t) f & := - \frac{\widehat{\kappa}}{2 \mathcal{G}} \left( \frac{\partial \mathcal{G} }{\partial X_\alpha} \right) \mathfrak{g}^{\alpha \beta } \frac{\partial f }{\partial X_\beta},\\
B_3 (t) f & := - \widehat{\kappa} \frac{\partial \mathfrak{g}^{\alpha \beta}}{\partial X_\alpha} \frac{\partial f }{\partial X_\beta},\\
B_4 (t) f & := - \frac{\partial \widehat{\kappa}}{\partial X_\alpha} \mathfrak{g}^{\alpha \beta} \frac{\partial f }{\partial X_\beta},\\
B_5 (t) f & := \frac{1}{2 \mathcal{G}} \left( \frac{d \mathcal{G}}{d t} \right) f.
\end{align*}
We first consider $B_1 (t) f$. From $( \mathfrak{g}^{\alpha \beta})_{2 \times 2} = ( \mathfrak{g}_{\alpha \beta})^{-1}_{2 \times 2}$, we find that
\begin{equation*}
B_1 (t) f = - \left( \widehat{\kappa} \frac{\mathfrak{g}_{22}}{\mathcal{G}} - \lambda_1 \right) \frac{\partial^2 f }{\partial X_1^2} - \left( \widehat{\kappa} \frac{\mathfrak{g}_{11}}{\mathcal{G}} - \lambda_2 \right) \frac{\partial^2 f }{\partial X_2^2} + 2 \widehat{\kappa} \frac{\mathfrak{g}_{12}}{\mathcal{G}} \frac{\partial^2 f }{\partial X_1 \partial X_2}.
\end{equation*}
This gives
\begin{multline}\label{eq76}
\| B_1 (t) f \|_{L^2 (U)} \\
\leq \left( \left\| \widehat{\kappa} \frac{\mathfrak{g}_{11} }{\mathcal{G}} - \lambda_2 \right\|_{L^\infty} + \left\| \widehat{\kappa} \frac{\mathfrak{g}_{22} }{\mathcal{G}} - \lambda_1 \right\|_{L^\infty} + 2 \left\| \widehat{\kappa} \frac{\mathfrak{g}_{12}}{ \mathcal{G}} \right\|_{L^\infty}  \right) \| \nabla_X^2 f \|_{L^2(U)}.
\end{multline}
Next we consider $B_2 (t) f + B_3 (t) f + B_4 (t) f$. From $( \mathfrak{g}^{\alpha \beta})_{2 \times 2} = ( \mathfrak{g}_{\alpha \beta})^{-1}_{2 \times 2}$, we find that
\begin{multline*}
B_2 (t) f + B_3 (t) f + B_4 (t) f\\
= \bigg\{  \frac{\widehat{\kappa}}{ \mathcal{G} } \left( \frac{\partial \mathfrak{g}_{22}}{\partial X_1} - \frac{\partial \mathfrak{g}_{12}}{\partial X_2} \right) - \frac{\widehat{\kappa}}{2 \mathcal{G}^2} \left( \mathfrak{g}_{22} \frac{\partial \mathcal{G} }{\partial X_1} - \mathfrak{g}_{12} \frac{\partial \mathcal{G} }{\partial X_2} \right) \bigg\} \frac{\partial f}{\partial X_1}\\
+ \bigg\{ \frac{\widehat{\kappa}}{ \mathcal{G} } \left( \frac{\partial \mathfrak{g}_{11}}{\partial X_2} - \frac{\partial \mathfrak{g}_{12}}{\partial X_1} \right) - \frac{\widehat{\kappa}}{2 \mathcal{G}^2} \left( \mathfrak{g}_{11} \frac{\partial \mathcal{G} }{\partial X_2} - \mathfrak{g}_{12} \frac{\partial \mathcal{G} }{\partial X_1} \right) \bigg\} \frac{\partial f}{\partial X_2}\\
+ \left( \frac{\mathfrak{g}_{22}}{\mathcal{G}} \frac{\partial \widehat{\kappa}}{\partial X_1} - \frac{\mathfrak{g}_{12}}{\mathcal{G}} \frac{\partial \widehat{\kappa}}{\partial X_2} \right) \frac{\partial f }{\partial X_1} + \left( \frac{\mathfrak{g}_{11}}{\mathcal{G}} \frac{\partial \widehat{\kappa}}{\partial X_1} - \frac{\mathfrak{g}_{12}}{\mathcal{G}} \frac{\partial \widehat{\kappa}}{\partial X_2} \right) \frac{\partial f }{\partial X_2} .
\end{multline*}
Therefore, we see that
\begin{equation}\label{eq77}
\| B_2 (t) f + B_3 (t) f + B_4 (t) \|_{L^2(U)} \leq ( \mathcal{M}_2(t) + \mathcal{M}_3(t) + \mathcal{M}_4(t)) \| \nabla_X f \|_{L^2}.
\end{equation}
It is easy to check that
\begin{equation}\label{eq78}
\| B_5 (t) f \|_{L^2(U)} \leq \mathcal{M}_5(t) \| f \|_{L^2 (U)}.
\end{equation}

Now we derive \eqref{eq72}. Using the interpolation inequality \eqref{eq34}, we find that there is $C_{\star \star} >0$ independent of $f$ such that for all $0 \leq t <T$,
\begin{multline}\label{eq79}
\| B_2 (t) f + B_3 (t ) f + B_4 (t) f \|_{L^2 (U)} \\
\leq \left( \left\| \widehat{\kappa} \frac{\mathfrak{g}_{11} }{\mathcal{G}} - \lambda_2 \right\|_{L^\infty} + \left\| \widehat{\kappa} \frac{\mathfrak{g}_{22} }{\mathcal{G}} - \lambda_1 \right\|_{L^\infty}  \right) \| \nabla_X^2 f \|_{L^2(U)} +  C_{\star \star} \| f \|_{L^2(U)}.
\end{multline}
By \eqref{eq34}, \eqref{eq76}, \eqref{eq78}, and \eqref{eq79}, we have
\begin{equation*}
\| B (t) f \|_{L^2 (U)} \leq 2 C_\sharp \mathcal{M}_1 \| A f \|_{L^2 (U)} + (C_{\star \star} + \mathcal{M}_5) \| f \|_{L^2 (U)},
\end{equation*}
which is \eqref{eq72}. Similarly, we use \eqref{eq34}, \eqref{eq76}, \eqref{eq77}, and \eqref{eq78} to have \eqref{eq73} and \eqref{eq74}.

Finally, we derive \eqref{eq75}. Fix $\varphi \in C ((0,T) ; D (A) )$ and $0 < s \leq t <T$. Since $\widehat{x} = \widehat{x}(X,t)$ is a $C^3$-function (Definition \ref{def21}), $\widehat{\kappa} = \widehat{\kappa} (X,t) $ is a $C^2$-function (Assumption \ref{ass23}), and
\begin{align*}
B (t) \varphi (t) - B (s) \varphi (s) = \{ B (t) - B (s) \} \varphi (t) + B (s) \{ \varphi (t) - \varphi (s) \},
\end{align*}
we use \eqref{eq32} and the mean-value theorem to have \eqref{eq75}. Therefore, Lemma \ref{lem71} is proved. 
\end{proof}

\subsection{Existence of Strong Solutions to Advection-Diffusion Equation}\label{subsec72}

Let us show the existence of strong solutions to the advection-diffusion equation \eqref{eq11} by Theorem \ref{thm63} and Lemmas \ref{lem51}-\ref{lem53}, \ref{lem71}, and \ref{lem72}. Let $C_\sharp$, $C_A$, and $C_\star$ be the three positive constants in \eqref{eq32}, \eqref{eq38}, and \eqref{eq72}. Define
\begin{equation*}
Z_T = \{ \varphi \in C ((0,T);L^2(U)); { \ } \| \varphi \|_{Z_T} <\infty \}
\end{equation*}
with
\begin{equation*}
\| \varphi \|_{Z_T} := \sup_{0 < t <T}\{ \mathrm{e}^{- t} \| \varphi \|_{L^2 (U)} \} + \| d \varphi /{d t} \|_{L^2 (0,T;L^2(U))} + \| A \varphi \|_{L^2 (0,T;L^2(U))}.
\end{equation*}
Assume that
\begin{equation}\label{Eq710}
2 C_\sharp \mathcal{M}_1 \leq \frac{1}{4}.
\end{equation}
From Lemmas \ref{lem34}, \ref{lem71} and \ref{lem72}, we see that the three operators $A$, $B$ and $L$ satisfy the properties as in Assumption \ref{ass61}. Therefore, we apply Theorem \ref{thm63} and Lemma \ref{lem71} to have the following three propositions.
\begin{proposition}\label{prop73}
Assume that $T < \infty$ and that
\begin{equation}\label{Eq711}
2 C_\sharp \mathcal{M}_1 ( C_A + 1 ) < \frac{1}{4 \sqrt{2}}.
\end{equation}
Then for each $v_0 \in D ( A^{\frac{1}{2}})$ system \eqref{eq71} admits a unique strong solution $v$ in $Z_{T_\star}$, satisfying
\begin{equation*}
\| v \|_{Z_{T_\star}} \leq 2 \| v_0 \|_{L^2(U)} + 2 \sqrt{2} \| A^{\frac{1}{2}} v_0 \|_{L^2(U)}.
\end{equation*}
Here
\begin{equation*}
T_\star = \min \left\{ T ,  \frac{1}{2} \log \left( 1 + \frac{1}{16 C_\star^2 (C_A + 1)^2} \right) \right\}.
\end{equation*}
\end{proposition}
\begin{proposition}\label{prop74}
Assume that $T < \infty$ and that
\begin{equation}\label{Eq712}
2 C_\sharp \left( \mathcal{M}_1 + \mathcal{M}_2 + \mathcal{M}_3 + \mathcal{M}_4 \right)  ( C_A + 1 ) < \frac{1}{4 \sqrt{2}}.
\end{equation}
Then for each $v_0 \in D ( A^{\frac{1}{2}})$ system \eqref{eq71} admits a unique strong solution $v$ in $Z_{T_*}$, satisfying
\begin{equation*}
\| v \|_{Z_{T_*}} \leq 2 \| v_0 \|_{L^2(U)} + 2 \sqrt{2} \| A^{\frac{1}{2}} v_0 \|_{L^2(U)}.
\end{equation*}
Here
\begin{equation*}
T_* = \min \left\{ T,  \frac{1}{2} \log \left( 1 + \frac{1}{16 \mathcal{M}_5^2 (C_A + 1)^2} \right) \right\}.
\end{equation*}
\end{proposition}

\begin{proposition}\label{prop75}
Assume that $T = \infty$ and that
\begin{equation}\label{Eq713}
2 C_\sharp ( \mathcal{M}_1 + \mathcal{M}_2 + \mathcal{M}_3 + \mathcal{M}_4+ \mathcal{M}_5 )  ( C_A + 1 ) < \frac{1}{4 \sqrt{2}}.
\end{equation}
Then for each $v_0 \in D ( A^{\frac{1}{2}} )$ system \eqref{eq71} admits a unique strong solution $v$ in $Y_\infty$, satisfying
\begin{equation*}
\| v \|_{Z_\infty} \leq 2 \| v_0 \|_{L^2(U)} + 2 \sqrt{2} \| A^{\frac{1}{2}} V_0 \|_{L^2(U)} .
\end{equation*}
\end{proposition}
\noindent Remark that \eqref{Eq710} holds if either \eqref{Eq711}, \eqref{Eq712}, or \eqref{Eq713} holds. Applying \eqref{eq35}, Lemmas \ref{lem51}-\ref{lem53}, Propositions \ref{prop73}-\ref{prop75}, we have Theorems \ref{thm24}-\ref{thm26}.

\section{Appendix: Maximal $L^p$-Regularity and Semigroup Theory}\label{sect8}

In this section we introduce the maximal $L^p$-in-time regularity for Hilbert space-valued functions and the basic semigroup theory. Let $H$ be a (complex) Hilbert space and $\| \cdot \|_H$ its norm. Let $\mathcal{A}: D (\mathcal{A}) (\subset H) \to H$ and $\mathcal{B}: D (\mathcal{B}) \to H$ be two linear operators on $H$ such that $D (\mathcal{A}) \subset D (\mathcal{B})$. Define $\mathcal{L} f := \mathcal{A} f + \mathcal{B} f $ and $D (\mathcal{L}) := D (\mathcal{A})$. 

We first state the maximal $L^p$-regularity of the generator of a bounded analytic semigroup on $H$. From \cite{Des64}, we have the following proposition.
\begin{proposition}[Maximal $L^p$-regularity]\label{prop81}
Let $1 < p < \infty$. Assume that $- \mathcal{A}$ generates a bounded analytic semigroup on $H$. Then $\mathcal{A}$ has the maximal $L^p$-regularity, i.e., there is $C_{\mathcal{A}} (p) >0$ such that for each $T \in (0, \infty ]$ and $F \in L^p (0, T ; H ) $ there exists a unique function $V$ satisfying the system:
\begin{equation*}
\begin{cases}
\frac{d}{d t} V + \mathcal{A} V = F & \text{ on } (0,T),\\
V|_{t =0} = 0,
\end{cases}
\end{equation*}
and the estimates:
\begin{equation*}
\| d V/{d t} \|_{L^p(0,T;H)} + \| \mathcal{A} V \|_{L^p(0,T;H)} \leq C_{\mathcal{A}}(p) \| F \|_{L^p(0,T;H)} .
\end{equation*}
\end{proposition}
See also \cite{DHP03} and \cite{KW04} for maximal $L^p$-regularity.

Next we introduce basic semigroup properties. From \cite[Sections 2 and 3]{Paz83} and \cite[Section 2]{Tan79}, we have the following two propositions.
\begin{proposition}[Perturbation theory]\label{prop82}{ \ }\\Assume that $- \mathcal{A}$ generates an analytic semigroup on $H$. Suppose that there are $C_1, C_2 >0$ such for all $f \in D (\mathcal{A})$
\begin{equation*}
\| \mathcal{B} f \|_H \leq C_1 \| \mathcal{A} f \|_H + C_2 \| f \|_H.
\end{equation*}
Then the following two assertions hold;\\
$(\mathrm{i})$ There is $\delta_{\mathcal{A}} >0$ such that if $C_1 \leq \delta_{ \mathcal{A} }$ then the operator $- \mathcal{L}$ generates an analytic semigroup on $H$.\\
$(\mathrm{ii})$ Assume in addition that $- \mathcal{A}$ generates a contraction $C_0$-semigroup on $H$. Then the operator $- \mathcal{L}$ generates an analytic semigroup on $H$ if $C_1 \leq 1/4$.
\end{proposition}

\begin{proposition}[Fractional power of $\mathcal{A}$]\label{prop83}{ \ }\\Assume that $- \mathcal{A}$ generates a bounded analytic semigroup on $H$ and that $\mathcal{A}$ is a non-negative selfadjoint operator on $H$. Then the following three assertions hold:\\
$(\mathrm{i})$ For each $0 < \mathfrak{a} \leq 1$, $0 < t <T$, and $f \in D (\mathcal{A}^{\mathfrak{a}})$,
\begin{equation*}
 \mathcal{A}^{\mathfrak{a}} \mathrm{e}^{- t \mathcal{A}} f =  \mathrm{e}^{- t \mathcal{A}} \mathcal{A}^{\mathfrak{a}} f.
\end{equation*}
$(\mathrm{ii})$ For each $0 < \mathfrak{a} \leq 1$, there is $C_{\mathfrak{a}} >0$ such that for all $t>0$ and $f \in D (\mathcal{A}^{\mathfrak{a}})$,
\begin{equation*}
\| \mathcal{A}^{\mathfrak{a}} \mathrm{e}^{- t \mathcal{A}} f \|_H \leq C_{\mathfrak{a}} t^{- \mathfrak{a}} \| f \|_H.
\end{equation*}
$(\mathrm{iii})$ For each $0 < \mathfrak{a} \leq 1$, there is $C_{\mathfrak{a}} >0$ such that for all $t>0$ and $f \in D (\mathcal{A}^{\mathfrak{a}})$,
\begin{equation*}
\| ( \mathrm{e}^{- t \mathcal{A}} - 1 ) f \|_H \leq C_{\mathfrak{a}} t^{\mathfrak{a}} \| \mathcal{A}^{\mathfrak{a}} f \|_H.
\end{equation*}
\end{proposition}

\end{document}